%% file: matrix_gens.tex
\documentclass[a4paper]{amsart}

\input{amsart_header_18_05}

\usepackage[foot]{amsaddr}

\numberwithin{equation}{section} 

\usepackage{enumitem} 

\newcommand{\M}[1][n]{\mathrm{M}_{{#1}}}
\renewcommand{\GL}[1][n]{\mathrm{GL}_{#1}}
\newcommand{\PGL}[1][n]{\mathrm{PGL}_{#1}}
\newcommand{\Ga}{{\mathbb{G}}_{\mathit{a}}}
\newcommand{\Gm}{{\mathbb{G}}_{\mathit{m}}}

\DeclareMathOperator{\Sub}{Sub}
\newcommand{\indep}[1][]{\mathrm{indep}_{{#1}}}

\title{Irredundant Generating Sets for Matrix Algebras}

\author{Yonatan Blumenthal$^*$}
\address{$^*$University of Haifa}
\email{yonatan@blumenthal.org.il}

\author{Uriya A.\ First$^*$}
\email{ufirst@univ.haifa.ac.il}

\begin{document}

\maketitle

\begin{abstract}
Let $F$ be a field.
We show that the largest irredundant generating
sets for the algebra of $n\times n $
matrices over $F$ have $2n-1$ elements when $n>1$.
(A result of Laffey states that the answer is $2n-2$
when $n>2$, but its proof contains an error.)
We further give a classification of the largest
irredundant generating sets when $n\in\{2,3\}$
and $F$ is algebraically closed.
We use this description to compute the dimension of
the variety
of $(2n-1)$-tuples of $n\times n$ matrices
which form an irredundant generating set when $n\in\{2,3\}$, and draw 
some consequences to \emph{locally redundant} generation of Azumaya algebras.
In the course of   proving the classification, we also determine
the largest 
sets $S$ of subspaces of $F^3$ with the property
that every $V\in S$ admits a matrix stabilizing every subspace in $S-\{V\}$
and  not stabilizing $V$.
\end{abstract}

\section{Introduction}

Let $F$ be a field and $n\in\N$. A set of matrices
$S\subseteq \M(F)$ is called an \emph{irredundant
generating set} for $\M(F)$ if $S$ generates $\M(F)$
(as a unital $F$-algebra), and no proper subset of $S$
generates $\M(F)$. 
It is well-known that $\M(F)$ can be generated by $2$ elements. However,
for $n>1$, $\M(F)$ has larger irredundant generating subsets.
Our first  result  determines the size
of the largest ones.

\begin{thm}\label{TH:intro-main}
	Suppose $n>1$. The largest irredundant generating sets for $\M(F)$
	have $2n-1$ elements.
\end{thm}

We remark that since $\M(F)$ is finite  dimensional,
every generating set for $\M(F)$ contains a finite generating
set, which in turn contains  an irredundant generating set.
Thus, every generating set for $\M(F)$ contains a generating subset
with $2n-1$  or less elements.

An explicit example of an irredundant generating set for $\M(F)$
having $2n-1$ elements is $\{a_1,\dots,a_{n-1},a'_1,\dots,a'_{n-1},b\}$,
where
\begin{align}\label{EQ:irred-gen-ex}
a_1 &=
\left[
\begin{smallmatrix}
1 & 1 &   & & \\
0 & 0 &   & & \\
  &   & 0 & & \\
  &   &   & 0 & \\[-0.3em]
  &   &   &        & \ddots \\
  &   &   &        & & 0
\end{smallmatrix}
\right],
\quad
a_2 =
\left[
\begin{smallmatrix}
0 &  &   & & \\
 & 1 & 1  & & \\
  &  0 & 0 & & \\
  &   &   & 0 & \\[-0.3em]
  &   &   &        & \ddots \\
  &   &   &        & & 0
\end{smallmatrix}
\right],
\quad
\dots,
\quad
a_{n-1} =
\left[
\begin{smallmatrix}
0 &  &   & & \\[-0.3em]
 & \ddots &    & & \\
  &    & 0 & & \\
  &   &   & 0 & \\
  &   &   &   & 1 & 1 \\
  &   &   &   & 0 & 0
\end{smallmatrix}
\right],
\\
a'_1 &=
\left[
\begin{smallmatrix}
0 & 0 &   & & \\
1 & 1 &   & & \\
  &   & 0 & & \\
  &   &   & 0 & \\[-0.3em]
  &   &   &        & \ddots \\
  &   &   &        & & 0
\end{smallmatrix}
\right],
\quad
a'_2 =
\left[
\begin{smallmatrix}
0 &  &   & & \\
 & 0 & 0  & & \\
  &  1 & 1 & & \\
  &   &   & 0 & \\[-0.3em]
  &   &   &        & \ddots \\
  &   &   &        & & 0
\end{smallmatrix}
\right],
\quad
\dots,
\quad
a'_{n-1} =
\left[
\begin{smallmatrix}
0 &  &   & & \\[-0.3em]
 & \ddots &    & & \\
  &    & 0 & & \\
  &   &   & 0 & \\
  &   &   &   & 0 & 0 \\
  &   &   &   & 1 & 1
\end{smallmatrix}
\right],
\nonumber
\\
b&=
\left[
\begin{smallmatrix}
1   \\
& 0   \\[-0.3em]
& & \ddots  \\
& & & 0  
\end{smallmatrix}
\right],
\nonumber
\end{align}
see Proposition~\ref{PR:irred-example}.
Theorem~\ref{TH:intro-main} contradicts a result of Laffey \cite{Laffey_1983_irred_gen_sets},
which states that irredundant generating sets
for $\M (F)$ have at most $2n-2$ elements when $n\geq 3$.
However, Laffey's proof   has a mistake, 
see Remark~\ref{RM:mistake}.
Nevertheless, most of Laffey's work remains intact, and we build on some of his
results in order to prove Theorem~\ref{TH:intro-main}. 

\medskip

Our second main result 
is a classification
of the largest irredundant generating sets of $\M(F)$
when $n=3$ and $F$ is algebraically closed. 
We also prove a similar result for $n=2$,
see Proposition~\ref{PR:max-gen-2}, but restrict here to $n=3$
for brevity. 

To phrase our result, 
consider the following operations
on subsets of $\M(F)$:
\begin{enumerate}[label=(\arabic*)]\label{list:operations}
\item conjugating all elements in the set by some $g\in \GL[n](F)$;
\item transposing all elements in the set;
\item replacing a matrix $a$ in the set with $\alpha a+\beta I$ for some $\alpha\neq 0$ and $\beta$ in $F$.
\end{enumerate}
For $S,T\subseteq \M(F)$,
we write $S\sim T$ to denote that $S$ can be transformed into $T$
using operations (1)--(3). Observe that in this case, $S$
is an irredundant generating set if and only if $T$ is.
We show that modulo operations (1)--(3), the $5$-element irredundant generating sets
of $\M[3](F)$ reduce into a $1$-parameter family.

\begin{thm}\label{TH:five-elem-irred-classification}
Suppose that $F$ is algebraically closed.
Then a $5$-element subset $S$ of $\M[3](F)$ is an irredundant generating set 
if and only if there is $\alpha\in F$
for which $S\sim S_\alpha$, where
\[
S_{\alpha} :=\left \{
\left[
\begin{smallmatrix}
1 & 1 & 0 \\
0 & 0 & 0 \\
0 & 0 & 0 
\end{smallmatrix}
\right],
\left[
\begin{smallmatrix}
0 & 0 & 0 \\
1 & 1 & 0 \\
0 & 0 & 0 
\end{smallmatrix}
\right],
\left[
\begin{smallmatrix}
0 & 0 & 0 \\
0 & 1 & 1 \\
0 & 0 & 0 
\end{smallmatrix}
\right],
\left[
\begin{smallmatrix}
0 & 0 & 0 \\
0 & 0 & 0 \\
0 & 1 & 1 
\end{smallmatrix}
\right],
\left[
\begin{smallmatrix}
1 & 0 & 0 \\
0 & 0 & 0 \\ 
0 & 0 & \alpha
\end{smallmatrix}
\right]
\right\}.
\]
\end{thm}

The generating sets of $\M[3](F)$ with less than $5$ elements were
described in  \cite{Aslaksen_2009_generators_mat_algs} in terms
of    (non-)vanishing of certain polynomial functions.

Returning to Theorem~\ref{TH:five-elem-irred-classification}, note that the set $S_{0}$ is precisely $\{a_1,a_2,a'_1,a'_2,b\}$,
where the notation is as in \eqref{EQ:irred-gen-ex} with $n=3$.
The element $\alpha$ from the theorem is not unique in general, but can attain
only finitely many values, see Remark~\ref{RM:no-of-alphas}.

As an application of Theorem~\ref{TH:five-elem-irred-classification},
we compute the dimension of the variety of $5$-tuples of $3\times 3$
matrices which form an irredundant generating set.
More formally,   suppose $F$ is algebraically closed 
and, for $r\in \N$, let $I_{r}^{(n)}$ be the set of tuples $(a_1,\dots,a_r)\in \M(F)^r$ 
which form an irredundant generating set for $\M(F)$.
Thinking of $ \M(F)^r$ as an affine space over $F$,
it turns out that the set $I_r^{(n)}$ is Zariski-locally-closed, and therefore has the structure
of an algebraic variety over $F$; see Section~\ref{sec:application}.
Operations (1)--(3) can now be packed into an action of the algebraic group
\[
G=(S_2\ltimes \PGL(F))\times (S_r\ltimes (\Gm\ltimes \Ga)^r)
\]
on the variety $I_r^{(n)}$. Specifically, $\PGL(F)=\GL(F)/\units{F}$
acts by simultaneous conjugation (operation (1)), the notrivial element
of $S_2$ acts as simultaneous transposition (operation (2)), $(\Gm\ltimes \Ga)^r$ acts
by $(\alpha_i,\beta_i)_{i=1}^r\cdot(a_1,\dots,a_r)=
(\alpha_1 a_1+\beta_1I,\dots,\alpha_r a_r+\beta_r I)$ (operation (3)),
and $S_r$ acts by permuting the coordinates of the $r$-tuple. 
Theorem~\ref{TH:five-elem-irred-classification}
is equivalent to saying that when $n=3$, there is a surjective $G$-equivariant morphism
$(g, \alpha)\mapsto g s_\alpha : G\times \bbA^1\to I_5^{(3)}$,
where $s_\alpha$ is the $5$-tuple 
\[
	s_\alpha:=
\left(	
	\left[
	\begin{smallmatrix}
1 & 1 & 0 \\
0 & 0 & 0 \\
0 & 0 & 0 
\end{smallmatrix}
\right],
\left[
\begin{smallmatrix}
0 & 0 & 0 \\
1 & 1 & 0 \\
0 & 0 & 0 
\end{smallmatrix}
\right],
\left[
\begin{smallmatrix}
0 & 0 & 0 \\
0 & 1 & 1 \\
0 & 0 & 0 
\end{smallmatrix}
\right],
\left[
\begin{smallmatrix}
0 & 0 & 0 \\
0 & 0 & 0 \\
0 & 1 & 1 
\end{smallmatrix}
\right],
\left[
\begin{smallmatrix}
1 & 0 & 0 \\
0 & 0 & 0 \\ 
0 & 0 & \alpha
\end{smallmatrix}
\right]
	\right).
	\]
Using this and the fact that the morphism has finite fibers,
we show in Proposition~\ref{PR:dim-Ir} that $\dim I_5^{(3)}=19$.
We also show that $\dim I_3^{(2)}=9$, but we do not know what is $\dim I^{(n)}_{r}$
in general.

For comparison, the set $Z_r^{(n)}$ of \emph{non-generating} $r$-tuples in $\M(F)^r$ 
also has the structure of an algebraic variety and its
dimension is $rn^2-(r-1)(n-1)$  \cite[Prop.~7.1(b)]{First_2022_generators_of_alg_over_comm_ring}.
In particular,   $\dim Z_5^{(3)}=37$, which informally means that
irredundant $5$-element generating sets for $\M[3](F)$ are   more ``rare'' than $5$-element non-generating sets.

Our computation of $\dim I_5^{(3)}$ has some consequences to generation of \emph{Azumaya algebras}.
Recall that an algebra $A$ over a ring $R$ is called Azumaya of degree $n$
if $A$ is a finitely generated projective $R$-module, and for every
maximal ideal $\frakm$ of $R$, the $R/\frakm$-algebra $A(\frakm):=A/\frakm A$
is a central simple of dimension $n^2$; for other equivalent definitions
and further details, see \cite[III.5.1.1]{Knus_1991_quadratic_hermitian_forms} and \cite[Chp.~7]{Ford_2017_separable_algebras}. Suppose that $F$ is any infinite 
field and $R$ is a finitely generated (commutative) $F$-ring
of Krull-dimension $d$. In \cite[Thm.~1.5(a)]{First_2022_generators_of_alg_over_comm_ring},
the second author, Reichstein and Williams showed that every Azumay algebra of degree $n$
over $R$ can be generated by $\floor{\frac{d}{n-1}}+2$
elements.
In particular, when $d\leq 7$, every Azumaya algebra $B$ of degree $3$
over $R$ can be generated by $5$ elements. We use results
from \cite{First_2022_generators_of_alg_over_comm_ring} and our computation that $\dim I_5^{(3)}=19$
to slightly improve this and show that
under the same hypotheses, $B$ admits a   generating $5$-tuple
$(b_1,\dots,b_5)\in B^5$
that is \emph{locally redundant} in
the sense that  for every maximal ideal $\frakm\idealof R$,
the images of $b_1,\dots,b_5$ in $B/\frakm B$ redundantly generates it as an $R/\frakm$-algebra; see
Corollary~\ref{CR:local-redundant-3}.\footnote{
By \cite[Lems.~2.1, 2.2]{First_2017_number_of_generators}, a generating tuple $(b_1,\dots,b_r)\in B^r$ is locally
redundant if and only if  
it is redundant 
Zariski-locally, i.e., $\Spec R$ admits a Zariski open covering
$\{U_i\}_{i\in I}$ such that the restrictions of $b_1,\dots,b_r$ to each $U_i$
redundantly generate
$B|_{U_i}$.}
Computation of $\dim I_r^{(n)}$ for other values of $n$ and $r$ will lead to similar results about Azumaya
algebras of other degrees;
see Theorem~\ref{TH:dim-Ir-conseq}.

\begin{remark}
	Fix $n>1$. 
	For $m>2$, we do not know if there exists a degree-$n$ Azumaya algebra
	admitting an $m$-element generating set and no $m$-element locally redundant generating
	set.
	If such an example were to exist,
	then it would also be an example of an Azumaya algebra that can be generated by
	$m$ elements and no less (because we can turn any generating set with $m-1$
	elements into a [locally] redundant generating set by adding another element).
	However, while there are examples of degree-$n$ Azumaya algebras
	requiring arbitrarily many elements to generate
	\cite[Thm.~1.5(b)]{First_2022_generators_of_alg_over_comm_ring}, \cite{Gant_2024_speaces_of_gens_2_by_2},
	there is no particular $m>2$ for which it is known
	if there is an Azumaya algebra that can be generated by $m$ elements and no less.
\end{remark}

While Theorem~\ref{TH:five-elem-irred-classification} implies
Theorem~\ref{TH:intro-main} in the case $n=3$, its proof
takes an entirely different path, which leads to other  questions
and byproducts.
In order to explain it, we   introduce some terminology.

Let $G$ be a group acting on a set
$X$. We call a subset $Y\subseteq X$
\emph{$G$-dependent} if there is a proper subset
$Y_0\subsetneq Y$ such that any $g\in G$
fixing every element of $Y_0$
also fixes every element of $Y$; otherwise, we call
$Y$ \emph{$G$-independent}. Equivalently,
a subset $Y\subseteq X$ is $G$-independent
if  for every $y\in Y$, there is $g\in G$
fixing $Y-\{y\}$ element-wise and not fixing $y$.
The size of the largest-possible $G$-independent subset
of $X$ is denoted $\indep[G] X$, or $\indep  X$,
and  called the \emph{$G$-independence number} of $X$.
For example, if $G=\GL[n](F)$ and $X=F^n$,
then a subset $Y\subseteq F^n$ is $\GL[n](F)$-independent
if and only if it is linearly independent.

Suppose now that $F$ is algebraically
closed and consider the action of $\GL[n](F)$ on the set $\mathrm{Sub}(F^n)$
of subspaces of $F^n$. 
It is an easy consequence of Burnside's theorem on matrix algebras
that every irredundant generating set  $S$ for  $\M(F)$ gives
rise to a $\GL(F)$-independent subset of $\Sub(F^n)$
of the same cardinality; see Proposition~\ref{PR:indep-num-and-irred-gen}.
(Briefly,  Burside's theorem implies 
that for every $a\in S$, there is subspace $V_a\subseteq F^n$
that is invariant under all members of $S-\{a\}$ and 
but not under $a$. Consider the set of all the $V_a$.) 
Thus, every irredundant generating set of $\M(F)$
has at most $\indep[\GL(F)]\Sub(F^n)$ elements.
The example in \eqref{EQ:irred-gen-ex} now implies that  
$\indep[\GL(F)]\Sub(F^n)\geq 2n-1$ and equality would 
imply Theorem~\ref{TH:intro-main}.

We show in Theorem~\ref{TH:GL-indep-3} and Proposition~\ref{PR:GL-indep-2}
that when   $n\in\{2,3\}$ and $F$ is any field with more than $2$ elements, we indeed have 
\[\indep[\GL(F)]\Sub(F^n)=2n-1.\]  
Moreover, we  
classify all the $(2n-1)$-element $\GL(F)$-independent subsets in these cases.\footnote{
	In fact, it was  via   proving these results that  we   realized that
	there is a mistake in \cite{Laffey_1983_irred_gen_sets}, and thus set to fix it Theorem~\ref{TH:intro-main}.
}
Specifically, a $5$-element subset of $\Sub(F^3)$
is $\GL[3](F)$-independent if and only if it consists of two $1$-dimensional 
subspaces $L_1,L_2$, two $2$-dimensional subspaces $P_1,P_2$ with $L_i\subseteq P_i$
($i=1,2$) and another nontrivial subspace
in general position  with respect to $L_1,L_2,P_1,P_2$.
It is thanks to this structural result that we can prove Theorem~\ref{TH:five-elem-irred-classification}.

The proof Theorem~\ref{TH:GL-indep-3} relies on analysing configurations of lines and planes
in $F^3$. This case-by-case approach is difficult
to apply  for larger values $n$, so we do not know what is $\indep[\GL(F)]\Sub(F^n)$ for $n>3$.

\medskip

We finish by noting that   the computation of $\indep[{\GL[3](F)}]\Sub(F^3)$
has other implications.
First, recall that $ \Sub(F^3)-\{0,F^3\}$ is the set of vertices
of the \emph{spherical building} of $\GL[3](F)$ (equiv.\ $\mathrm{SL}_3(F)$
or $\PGL[3](F)$); see \cite[\S6.5]{Abramenko_2008_Buildings}.
This raises the question of what is the independence number of the action of
other reductive
algebraic groups acting on their spherical building? 
Our results show that the answer depends on   which   types of faces in the building are
considered.
Indeed, in our case of $\GL[3](F)$,
the type  of  $V\in\Sub(F^3)-\{0,F^3\}$ is its dimension as a vector space.
The classification of the largest $\GL[3](F)$-independent sets 
in $\Sub(F^3)$ shows that they include both $1$-dimensional and $2$-dimensional
subspaces, so the $\GL[3](F)$-independence number of the set of $1$-dimensional subspaces
of $F^3$ is at most $4$ (and not $5=\indep[{\GL[3](F)}]\Sub(F^3)$); in fact, it is exactly $4$
(Remark~\ref{RM:four-lines-indep}).

Second, for a general group $G$ acting on a set $X$,
the fact that $S\subseteq X$ is $G$-indepednent may be phrased in terms
of stabilizers, namely, $S$ is $G$-independent if and only if
for every
$x\in S$, we have $\Stab_G(x)\nsupseteq \bigcap_{y\in S-\{x\}}\Stab_G(y)$
(when $S-\{x\}=\emptyset$, the intersection is understood as $G$).
When $F$ is algebraically
closed, the stabilizers of members of $\Sub(F^3)-\{0,F^3\}$ are precisely
the maximal parabolic subgroups of $\GL[3](F)$. Taking that
as the definition of maximal parabolic subgroups 
when $F$ is general, our Theorem~\ref{TH:GL-indep-3} implies
the following group-theoretic result.

\begin{cor}
	Let $F$ be a  field with $|F|>2$ and let $P_1,\dots,P_t$ be maximal
	parabolic subgroups of  $\GL[3](F)$. Suppose
	that for every $i\in\{1,\dots,t\}$, we have $P_i\nsupseteq \bigcap_{j\neq i}P_j$.
	Then $t\leq 5$.
\end{cor}

The paper is organized as follows.
Section~\ref{sec:prelim}   sets some general notation for the paper.
Theorem~\ref{TH:intro-main} is proved in Section~\ref{sec:irred-Mn}.
Section~\ref{sec:GL-indep} discusses $\GL(F)$-indepednent
sets in $\Sub(F^n)$ and their connection to irredundant generating
sets in $\M(F)$, and uses it to prove an analogue of Theorem~\ref{TH:five-elem-irred-classification}
for $2\times 2$ matrices. Theorem~\ref{TH:five-elem-irred-classification} is then
proved in Section~\ref{sec:three-by-three}.
Finally, in Section~\ref{sec:application} we compute the dimensions of $I^{(3)}_5$
and $I^{(2)}_3$ and derive some applications to locally redundant generation of Azumaya algebras.

\subsection*{Acknowledgments}

We are grateful to Omer Cantor for  comments about an earlier version of this manuscript.
The second author is supported by an ISF grant no.\ 721/24.

\section{Conventions and Preliminaries}
\label{sec:prelim}

Throughout this paper, $F$ is a field.

If not indicated otherwise,  rings are commutative and unital. 
Algebras are unital and associative by default, but are
not necessarily commutative. If $A$ is an $F$-algebra
and $S\subseteq A$, we say that $S$ generates $A$ if it generates $A$
as a unital algebra. The \emph{non-unital} algebra generated by $S$
will be denoted $\Trings{S}$.
An $F$-ring, resp.\ $F$-field, is an $F$-algebra which is a ring, resp.\ field.
The set of maximal ideals of a ring $R$ is denoted $\Max R$.

Recall the following known facts about generating sets.

\begin{prp}[{\cite[Lem.~1.1]{Laffey_1983_irred_gen_sets}} or {\cite[Lem.~2.2]{First_2022_generators_of_alg_over_comm_ring}}]
	\label{PR:unital-gen}
	Let $n>1$. A subset of $ \M(F)$ generates it as unital algebra if
	and only if it generates it as a non-unital algebra.
\end{prp}

\begin{prp}\label{PR:field-ext-gen}
	Let $A$ be an $F$-algebra,
	let $F'$ be an $F$-field,
	let $S\subseteq A$ and let
	$S'$ be the image of $S$ in $A':=A\otimes_F F'$.
	Then $S$ generates $A$ as an $F$-algebra if and only
	if $S'$ generates $A'$ as an $F'$-algebra.
\end{prp}

\begin{proof}[Proof (cf.\ {\cite[Lem.~1.2]{Laffey_1983_irred_gen_sets}})]
	Let $M$ be the the set of monomials in elements of $S$. The the image of $M$
	in $A'$, denoted $M'$, is the set of monomials in $S'$.
	Now,  $S$ generates $A$
	if and only if $\Span_F(M)=A$ if and only if $\Span_{F'}(M')=A'$
	if and only if $S'$ generates $A'$. 
\end{proof}

When there is no risk of confusion,
we abbreviate $\nMat{F}{n}$ to $\M$,
$\nMat{F}{m\times n}$ to $\M[m\times n]$ and $\nGL{F}{n}$
to $\GL$. 
Vectors in $F^n$ are viewed as column vectors and we freely
identify $\M$ with the $F$-algebra of $F$-endomorphisms of $F^n$.

The $n\times n$ identity matrix is denoted $I_n$, or just $I$.
The transpose of a matrix $a$ is   $a^\trans$.
When $n$ is clear from the context, we write $e_{i,j}$ for the $n\times n$ matrix with $1$ in the $(i,j)$
entry and zeroes elsewhere.
Given matrices $a=(\alpha_{ij})_{i,j}\in\M[m\times n]$
and $b\in \M[p\times q]$, their tensor product
is the the $mp\times nq$ matrix given in block form as follows
\[
a\otimes b
=\begin{bmatrix}
\alpha_{11} b & \cdots & \alpha_{1n} b \\
\vdots & & \vdots \\
\alpha_{m1} b & \cdots & \alpha_{mn}b
\end{bmatrix}.
\]

\section{Irredundant Generating Sets for $\nMat{F}{n}$; Proof of Theorem~\ref{TH:intro-main}}
\label{sec:irred-Mn}

In this section, we prove Theorem~\ref{TH:intro-main}.
We break the proof into two parts --- Proposition~\ref{PR:irred-example},
which says that $\M$ has an irredundant generating set with ${2n-1}$
elements when $n>1$, and Theorem~\ref{TH:irredundant-Mn},
which says that every irredundant generating set for $\M$
has at most $2n-1$ elements.

Recall that given a subset $S\subseteq \M$, we write $\Trings{S}$
for the non-unital algebra generated by $S$.
Given $k,m\in\N$, we further let
\[
U_{k,m}=\begin{bmatrix}
\M[k] & \M[k\times m] \\
0 & \M[m]
\end{bmatrix}
\qquad
\text{and}
\qquad
L_{k,m}=\begin{bmatrix}
\M[k] & 0 \\
\M[m\times k] & \M[m]
\end{bmatrix}.
\]
Both are (unital) subalgebras of $\M[k+m]$.

\begin{prp}\label{PR:irred-example}
Let $n>1$, and 
define
$a_1,\dots,a_{n-1},a'_1,\dots,a'_{n-1},b\in\M$
as in \eqref{EQ:irred-gen-ex}.
Then $S:=\{a_1,\dots,a_{n-1},a'_1,\dots,a'_{n-1},b\}$ is an irredundant
generating set for $\M$ consisting of $2n-1$ elements.
\end{prp}

\begin{proof}
	We first show that $\Trings{S}=\M$.
	To that end, observe that if $e_{i,i}\in\Trings{S}$
	for some $i\in \{1,\dots,n-1\}$,
	then   $ e_{i,i+1},e_{i+1,i},e_{i+1,i+1}\in\Trings{S}$,
	because $e_{i,i+1}=a_i-e_{i,i}$, $e_{i+1,i}=a'_i e_{i,i}$
	and $e_{i+1,i+1}=a'_i-e_{i+1,i}$.
	Since $e_{1,1}=b\in\Trings{S}$,
	it follows that $e_{i,j}\in \Trings{S}$
	whenever $|i-j|\leq 1$. The elements $e_{i,j}$ with $|i-j|\leq 1$ are well-known to generate
	$\M$, so $\Trings{S}=\M$.
	To see that $S$ is irredundant, observe that
	$S-\{a_i\}\subseteq L_{i,n-i}$,
	 $S-\{a'_i\}\subseteq U_{i,n-i}$,
	and $S-\{b\}$ is contained in the subalgebra of $\M$ 
	consisting of matrices having
	$(1,-1,1,-1,\dots)^\trans\in F^n$
	as a (right) eigenvector.
\end{proof}

\begin{remark}\label{RM:mistake}
	The main result of \cite{Laffey_1983_irred_gen_sets} states
	that for $n>2$, every irredundant generating set for $\M$
	has at most $2n-2$ elements, but
	Proposition~\ref{PR:irred-example} shows that this statement cannot be correct.
	
	Indeed, 
	there is a subtle error on the last paragraph of Section~3 in \cite{Laffey_1983_irred_gen_sets}, which deals with the case $n=3$.	
	In more detail, the author considers matrices $X_{11},Y_{11},X_1,X_2\in\M[2]$ such
	that $X_{11},Y_{11}$ are linearly independent, $X_1,X_2,I$
	are linearly independent and $\Trings{X_1,X_2,X_{12}Y_{21}}=\M[2]$
	for some nonzero $X_{12}\in\M[2\times 1]$ and $Y_{21}\in\M[1\times 2]$.
	It is claimed (via reduction to the case where $X_{11}\in FY_{11}$ that we could not
	understand) 
	that the $3\times 3$
	matrix 
	$[\begin{smallmatrix}
	X_{12}Y_{21} & 0 \\
	0 & 0 
	\end{smallmatrix}]$
	is in $A:=\Trings{[\begin{smallmatrix}
	X_1 & 0 \\
	0 & 0 
	\end{smallmatrix}],
	[\begin{smallmatrix}
	X_2 & 0 \\
	0 & 0 
	\end{smallmatrix}],
	[\begin{smallmatrix}
	X_{11} & X_{12} \\
	0 & 0 
	\end{smallmatrix}],
	[\begin{smallmatrix}
	Y_{11} & 0 \\
	Y_{21} & 0 
	\end{smallmatrix}],I}$.
	Unfortunately, this is not always true, e.g.,
	take $X_{11}=[\begin{smallmatrix}
	0 & 0 \\
	0 & 1 
	\end{smallmatrix}]$,
	$Y_{11}=[\begin{smallmatrix}
	1 & 0 \\
	0 & 1 
	\end{smallmatrix}]$,
	$X_1=[\begin{smallmatrix}
	1 & -1 \\
	0 & 0 
	\end{smallmatrix}]$,
	$X_2=[\begin{smallmatrix}
	0 & 0 \\
	-1 & 1 
	\end{smallmatrix}]$ ,
	$X_{12}=[\begin{smallmatrix}
	0 \\
	-1 
	\end{smallmatrix}]$
	and $Y_{21}=[\begin{smallmatrix}
	0 & 1 
	\end{smallmatrix}]$.
	One readily checks that every
	matrix in $A$ has $[\begin{smallmatrix} 1& 1&1\end{smallmatrix}]^\trans$
	as an eigenvector, but it is not an eigenvector
	of $[\begin{smallmatrix}
	X_{12}Y_{21} & 0 \\
	0 & 0 
	\end{smallmatrix}]$.

	Nevertheless, this is the only mistake in \cite{Laffey_1983_irred_gen_sets}
	that we are aware of, and as far as we can tell,
	all results in \cite{Laffey_1983_irred_gen_sets} which do not
	rely on the case $n=3$ are correct. 
	In fact, we use some of these results  in order to prove Theorem~\ref{TH:intro-main}.
\end{remark}

With Proposition~\ref{PR:irred-example} already proven,  Theorem~\ref{TH:intro-main}
reduces into proving:

\begin{thm}\label{TH:irredundant-Mn}
	Let $n\in\N$. Every irredundant generating set for $\M$ has $2n-1$ or less elements.
\end{thm}

We prove the theorem using some lemmas and two results of Laffey \cite{Laffey_1983_irred_gen_sets}.
In what follows,
we shall think of $\M[p\times q]$ ($p,q\leq n$) as  a non-unital subalgebra
of $\M$ via identifying $a\in \M[p\times q]$ with the block matrix
$[\begin{smallmatrix} a & 0 \\ 0 & 0\end{smallmatrix}]\in\M$.
Given $S\subseteq \M$
and $x\in \GL$,
we write  $xSx^{-1}$ for $\{xax^{-1}\where a\in S\}$
and $S^\trans$ for $\{a^\trans\where a\in S\}$.
Observe that $\Trings{S^\trans}=\Trings{S}^\trans$
and $\Trings{xSx^{-1}}=x\Trings{S}x^{-1}$.

\begin{lem}\label{LM:col-inc}
Let $1\leq p\leq q< n$ be integers and let $a\in \M$.
Suppose that $a\notin L_{q,n-q}$.
Then there exists $x\in  \GL $ such that $x\Trings{\M[p\times q]\cup \{a\}}x^{-1}\supseteq
\M[p\times(q+1)]$ and $x\M[p\times q]x^{-1}=\M[p\times q]$.
\end{lem}

\begin{proof}
There are $i,j\in \N$
such that $1\leq i\leq q<j\leq n$ and $a_{ij}\neq 0$.
Since $e_{1,i}\in \M[p\times q]$, we may replace $a$ with $e_{1,i}a$
to assume that  $i=1$, i.e.\ $a_{1j}\neq 0$, and $a$ has the form 
\[
a=
\begin{bmatrix}
u & v \\
0_{(n-1)\times q} & 0_{(n-1)\times (n-q)}
\end{bmatrix} 
\]
with $u\in\M[1\times q]$ and $v\in \M[1\times(n-q)]$.
Replacing $a$ with 
$a-[\begin{smallmatrix}
u & 0  \\
0  & 0 
\end{smallmatrix}]$, 
we may further assume that
$u=0$. 

Since $v\neq 0$, there is $Q\in \GL[n-q]$ such that
$vQ =[\begin{smallmatrix}
1 & 0 & \cdots & 0
\end{smallmatrix}]$. Put $x=I_q\oplus Q^{-1}$.
It is straightforward to check that $xax^{-1}=[\begin{smallmatrix}
0 & vQ \\
0 & 0
\end{smallmatrix}]=e_{1,q+1}$
and  
$x\M[p\times q]x^{-1}=\M[p\times q]$.
Now, for every $1\leq i\leq p$, we have $e_{i,q+1}=e_{i,1}e_{1,q+1}\in \Trings{\M[p\times q]\cup \{ xax^{-1}}\}$, so 
\[\M[p\times (q+1)]\subseteq
\Trings{\M[p\times q]\cup \{xax^{-1}\}}=
\Trings{x(\M[p\times q] \cup \{ a \})x^{-1}}
=x\Trings{\M[p\times q]\cup \{ a\}}x^{-1}.
\qedhere\]
\end{proof}

\begin{lem}\label{LM:row-inc}
Let $1\leq p< n$ be integers and let $a\in \M$.
Suppose that $a\notin U_{p,n-p}$.
Then there exists $x\in  \GL $ such that $x\Trings{\M[p\times n]\cup \{a\}}x^{-1}\supseteq
\M[(p+1)\times n]$  and $x\M[p\times n]x^{-1}=\M[p\times n]$. 
\end{lem}

\begin{proof}
There are $i,j\in \N$ such that $j\leq p<i\leq n$ such that $a_{ij}\neq 0$.
Since $e_{j,1}\in \M[p\times n]$, we may replace $a$ with $ae_{j,1}$ 
to assume that $j=1$, i.e., $a_{i1}\neq 0$,
and that $a$ has the form
\[
a=
\begin{bmatrix}
u & 0_{p\times (n-1)} \\
v & 0_{(n-p)\times (n-1)}
\end{bmatrix} 
\]
with $u\in\M[p\times 1]$, $v\in \M[(n-p)\times 1]$.
Replacing $a$ with
$a-[\begin{smallmatrix}
u & 0  \\
0  & 0 
\end{smallmatrix} ]$,
we may further assume that $u=0$.

Since $v\neq 0$, there is $P\in \GL[n-p]$ such that $Pv=[\begin{smallmatrix}
1 & 0 & \cdots & 0
\end{smallmatrix}]^\trans$. Put $x=I_p\oplus P$.
It is routine to check that
$xax^{-1}=[\begin{smallmatrix}
0 & 0 \\
Pv & 0
\end{smallmatrix} ]=e_{p+1,1}$
and   $x\M[p\times n]x^{-1}=\M[p\times n]$.
Now, for every $1\leq j\leq n$, we have
$e_{p+1,j}=e_{p+1,1}e_{1,j}\in \Trings{\M[p\times n]\cup \{xax^{-1}\}}$,
and as in the proof
of Lemma~\ref{LM:col-inc}, it follows that
$\M[(p+1)\times n]\subseteq  \Trings{\M[p\times n]\cup\{xax^{-1}\}}=x\Trings{\M[p\times n]\cup\{a\}}x^{-1}$.
\end{proof}

\begin{cor}\label{CR:completion}
Let $p,q\in \{1,\dots,n\}$
and let $S$ be a subset of $\M$
such that $\Trings{\M[p\times q]\cup S}=\M$.
Then there is $T\subseteq S$
with $|T|\leq 2n-p-q$ such that 
$\Trings{\M[p\times q]\cup T}=\M$.
\end{cor}

\begin{proof}
We argue by induction on $k:=(2n-p-q)$. The case $k=0$
is clear since $p=q=n$ and we can take $T=\emptyset$.
Suppose now that $k>0$ and the theorem is known to hold when $2n-p-q<k$.

Since $\Trings{\M[q\times p]\cup S^\trans}=\M$ if and only
if $\Trings{\M[p\times q]\cup S} =\M$,
the corollary holds for $\M[p\times q]$ and $S$ 
if and only if it holds for $\M[q\times p]$ and $S^\trans$. It is therefore enough
to consider the case where $p\leq q$.

Suppose first that $q<n$. Since 
$\Trings{\M[p\times q]\cup S}\nsubseteq
L_{q,n-q}$
and $\M[p\times q]\subseteq L_{q,n-q}$,
there is some $a\in S-L_{q,n-q}$.
By Lemma~\ref{LM:col-inc}, there is $x\in\GL$ such
that 
$\M[p\times (q+1)]\subseteq x\Trings{\M[p\times q]\cup \{a\}}x^{-1}$
and $x\M[p\times q]x^{-1}=\M[p\times q]$.
Observe that
\begin{align*}
\Trings{\M[p\times (q+1)]\cup xSx^{-1}} & \supseteq 
\Trings{ \M[p\times q] \cup xSx^{-1}}
\\
&=
\Trings{x\M[p\times q]x^{-1}\cup xSx^{-1}}=x\Trings{\M[p\times q]\cup S}x^{-1}=\M
.
\end{align*}
Thus, by the induction hypothesis, there is $R\subseteq xSx^{-1}$
with $|R|\leq 2n-p-q-1$ and such that
$\Trings{\M[p\times (q+1)],R}=\M$.
Put $T=\{a\}\cup x^{-1}Rx$. Then $T$ is a subset of $S$ with at most $2n-p-q$
elements and
\begin{align*}
\Trings{\M[p\times q]\cup T}
&=\Trings{\M[p\times q]\cup \{a\}\cup x^{-1}Rx}
\\
&\supseteq \Trings{x^{-1}\M[p\times (q+1)] x\cup x^{-1}R x}=
x^{-1}\Trings{\M[p\times(q+1)]\cup R}x=\M,
\end{align*}
so $T$ is the required subset of $S$.

Now suppose that $q=n$. Then $p<n$.
Since $\Trings{\M[p\times n]\cup S}\nsubseteq
U_{p,n-p}$
and $\M[p\times q]\subseteq U_{p,n-p}$,
there is some $a\in S-U_{p,n-p}$.
By Lemma~\ref{LM:row-inc}, there is $x\in\GL$ such
that 
$\M[(p+1)\times n]\subseteq x\Trings{\M[p\times n]\cup \{a\} }x^{-1}$
and $x\M[p\times n]x^{-1}=\M[p\times n]$.
Now, arguing as in the case $q<n$, we see that there
is $R\subseteq xSx^{-1}$
with $|R|\leq 2n-p-q-1$
and $\Trings{\M[(p+1)\times n]\cup R }=\M$,
and the set $T=\{a\}\cup x^{-1}R x$
satisfies all the requirements.
\end{proof}

We now recall two results from \cite{Laffey_1983_irred_gen_sets} that we will need.
Their proofs are unaffected by the mistake in  
\cite[\S3]{Laffey_1983_irred_gen_sets}, see
Remark~\ref{RM:mistake}.

\begin{thm}[{Laffey \cite[Thm.~4.1]{Laffey_1983_irred_gen_sets}}]
	\label{TH:Laffey-triang}
	Let $S$ be a irredundant generating set for $\M$ with $n>1$.
	If $|S|>n+1$, then $S$ contains a proper subset that is not triangularizable.
\end{thm}

\begin{construction}\label{CN:induction}
	Let $r_1,k_1,\dots,r_t,k_t\in\N$
	and put $n=r_1k_1+\dots+r_tk_t$.
	We write $(\M[r_i])_{k_i}$ for the subalgebra
	of $\M[r_ik_i]$ consisting of matrices of the form $a\oplus\dots\oplus a$
	($k_i$ times) with $a\in \M[r_i]$. Then $A:=(\M[r_1])_{k_1}\oplus\dots\oplus (\M[r_t])_{k_t}$
	is a semisimple unital subalgebra of $\M$.

	Fix some $x\in \M$. We shall write $x$ in block form  $(x_{ij})_{i,j}$
	with $x_{ij}$ being an $r_ik_i\times r_jk_j$ matrix.  
	For every $i,j\in\{1,\dots,t\}$,
	write $x_{ij}=(y_{uv})_{u,v}$ with $y_{uv}\in\M[r_i\times r_j]$,
	$u\in\{1,\dots,k_i\}$, $v\in\{1,\dots,k_j\}$.
	Choose a subset $B=\{b_1,\dots,b_q\}\subseteq\{y_{u,v}\}_{u,v}$
	which forms a basis to the $F$-span of the $\{y_{u,v}\}_{u,v}$.
	Then there are unique $m_1,\dots,m_q\in\M[k_i\times k_j]$
	such that
	\[
	x_{ij}=m_1\otimes b_1+\dots+m_q\otimes b_q.
	\]
	(Explicitly, for every $u,v$,
	there are unique $\mu_1^{u,v},\dots,\mu_{q}^{u,v}\in F$
	such that $y_{uv}=\mu_1^{u,v} b_1+\dots+\mu_q^{u,v} b_q$.
	The matrix $m_\ell$ is $(\mu_\ell^{u,v})_{u,v}$.)
	Put $\hat{x}_{ij}:=\{m_1,\dots,m_q\}$ if $B\neq \emptyset$
	and $\hat{x}_{ij}:=\{0_{k_i\times k_j}\}$ otherwise.
	Now, let
	$\hat{x}$ denote the set of matrices
	$w\in \M[k_1+\dots+k_t]$ such that when
	written in block form  $w=(w_{ij})_{i,j\in\{1,\dots,t\}}$
	with $w_{ij}\in \M[k_i\times k_j]$, one has
	$w_{ij}\in \hat{x}_{ij}$ for all $i,j\in\{1,\dots,t\}$.
	
	Finally, given a subset $S\subseteq \M$, let
	$\hat{S}=\bigcup_{x\in S}\hat{x}\subseteq \M[k_1+\dots+k_t]$.
\end{construction}

\begin{thm}[{Laffey \cite[Thm.~5.1]{Laffey_1983_irred_gen_sets}}]
	\label{TH:Laffey-ind}
	With notation as in Construction~\ref{CN:induction}, let
	$S\subseteq \M$. Then $\Trings{S\cup A}=\M$
	if and only if $\trings{\hat{S}\cup \hat{A}}=\M[k_1+\dots+k_t]$.
\end{thm}

We remark  that $\hat{A}$ is 
$\{a_1\oplus\dots\oplus a_t
\where
\text{$a_i\in\{0_{k_i},I_{k_i}\}$ for all $i$}\}$.
We are now ready to prove Theorem~\ref{TH:irredundant-Mn}.

\begin{proof}[Proof of Theorem~\ref{TH:irredundant-Mn}]
	Let $n\in\N$ and let $S $ be an irredundant generating set for $\M$.
	We need to prove that $|S|\leq 2n-1$.
	By Proposition~\ref{PR:field-ext-gen},  it is enough to show this when $F$ is algebraically
	closed.	We now  argue by induction on $n$.

	The case $n=1$ is clear,
	so assume 
	$n>1$.	
	If $|S|\leq n+1$, then we are done, so assume further that $|S|>n+1$.
	By Theorem~\ref{TH:Laffey-triang}, there is $T\subsetneq S$
	that is \emph{not} triangularizable. 
	Put $B=\Trings{T\cup \{I_n\}}$
	and let $J$ be its Jacobson radical.
	By  Wedderburn's Principal Theorem \cite[Thm.~2.5.37]{Rowen_1988_ring_theory_I},
	$B$ has a semisimple $F$-algebra $A$ such that $B={A\oplus J}$.
	Moreover, by the Artin--Wedderburn theorem and our assumption that
	$F$ is algebraically closed,
	there are $r_1,\dots,r_t\in\N$ such that $A\cong \M[r_1]\times \dots \M[r_t]$.
	After conjugating $S$ and $T$ by a suitable $x\in\GL$, we may
	further assume that
	$A$ is the subalgebra
	$(\M[r_1])_{k_1}\oplus\dots\oplus (\M[r_t])_{k_t}$
	of $\M$ as in Construction~\ref{CN:induction}.\footnote{
		To see the existence of $x$, view $F^n$ as a left module
		over $A\cong \M[r_1]\times \dots\times \M[r_t]$.
		There are $k_1,\dots,k_t$ such that $F^n\cong (F^{r_1})^{k_1}\times\dots\times (F^{r_t})^{k_t}$
		as left $ \M[r_1]\times \dots\times \M[r_t]$-modules. 
		After identifying $(F^{r_1})^{k_1}\times\dots\times (F^{r_t})^{k_t}$
		with $F^n$ in the usual way, the   isomorphism is  
		the $x$ we are looking for.	
	}
	Since $S$ is irredundant, $B\neq \M$ and therefore $r_i<n$ for all $i$.
	
	The algebra $\M[r_1]$ is an epimorphic image of $B=\Trings{T\cup\{I\}}$.
	Thus, by the induction hypothesis, there is $T_0\subseteq T$ 
	with $|T_0|\leq 2r_1-1$
	such that $\M[r_1]$ is an epimorphic image of $\Trings{ T_0\cup \{I\}}$.
	We replace $T$ with $T_0$ and redefine everything accordingly. This does not
	affect $r_1$, and we may now assume that 
	$(\M[r_1])_{k_1}\subseteq \Trings{ T\cup\{I\}}$ and $|T|\leq 2r_1-1$

	Suppose $k_1=1$. Then $\M[r_1]\subseteq \Trings{T\cup\{I\}}$.
	Since $\Trings{\M[r_1]\cup S}\supseteq \Trings{S}=\M$ (see Proposition~\ref{PR:unital-gen}
	for the second equality),
	Corollary~\ref{CR:completion}
	tells us that there is $S_0\subseteq S$ with $|S_0|=2n-2r_1$
	such that $\Trings{\M[r_1]\cup S_0}=\M$, so $\Trings{S_0\cup T\cup \{I\}}=\M$.
	As $S$ was irredundant, we conclude that $S=S_0\cup T $
	and   $|S|\leq (2n-2r_1)+(2r_1-1)=2n-1$.
	
	Finally,  suppose   $k_1>1$.
	Our assumption that $T$ is not triangularizable implies that $r_1>1$.
	Put $g:=k_1+\dots+k_t$ and observe that  
	\[ g\leq n-k_1(r_1-1)\leq n-2r_1+2<n.\]
	We have $\Trings{A\cup S}=\M$,
	so by Theorem~\ref{TH:Laffey-ind}, $\trings{\hat{A}\cup \hat{S}}=\M[g]$.
	By the induction hypothesis, there is $U\subseteq \hat{A}\cup \hat{S}$
	with $\Trings{U}=\M[g]$ and $|U|\leq 2g-1$.
	For every $y\in U$, choose some $x\in A\cup S$ with
	$y\in \hat{x}$ and let $S_0$ denote the set of $x$-s chosen this way
	that are not in $A$.
	Then  $\trings{\hat{A}\cup \hat{S}_0}=\M[g]$,
	and so another application of Theorem~\ref{TH:Laffey-ind}
	tells us that $\Trings{A\cup S_0}=\M$.
	Thus, $\Trings{T\cup \{I\}\cup S_0}\supseteq \Trings{A\cup S_0}=\M$,
	and again we must have $S=T\cup S_0$.
	But now 
	\begin{align*}
	|S|&\leq |T|+|S_0|\leq 2r_1-1+2g-1
	\\
	&\leq 
	2r_1-1+2(n-2r_1+2)-1 
	=2n-2r_1+2\leq 2n-1.\qedhere
	\end{align*}
\end{proof}

\section{$\GL$-Independence and the Case of $2\times 2$ Matrices}
\label{sec:GL-indep}

Consider the action on $\GL=\GL(F)$ on the set of subspaces of $F^n$,
denoted $\Sub(F^n)$.
Recall from the introduction that
a subset $S\subseteq \Sub(F^n)$ is called $\GL$-independent
if for every $V\in S$, there is $a\in \GL$ for which every
element of $S-\{V\}$ is $a$-invariant, but $V$ is not $a$-invariant.
(This means in particular that $0,V\notin S$.)
The   $\GL$-independence number of $\Sub(F^n)$,
denoted $\indep[\GL]\Sub(F^n)$, 
is the largest-possible size of a $\GL$-independent subset of $\Sub(F^n)$.

The following proposition relates $\indep[\GL]\Sub(F^n)$
and the size of irredundant generating
sets for $\M$.

\begin{prp}\label{PR:indep-num-and-irred-gen}
	Suppose $F$ is algebraically closed and $n>1$.
	If $S$ is an irredundant generating set for $\M$, then
	\[
	|S|\leq\indep[\GL]\Sub(F^n).
	\]
	More precisely, 
	one can construct from $S$ a $\GL$-independent   set $S'\subseteq \Sub(F^n)$
	with $|S'|= |S|$ as follows:
	for every $a\in S$, choose a subspace $V_a\subseteq F^n$
	that is invariant under every $b\in S-\{a\}$
	and not invariant under $a$. Then take $S'=\{V_a\where a\in S\}$.
\end{prp}

\begin{proof}
	Let $a\in S$ and put $S_a=S-\{a\}$.
	Since $S$ is irredundant,
	we have
	$\Trings{S_a\cup\{I\}}\subsetneq \M$.
	Burnside's theorem on matrix algebras (\cite[Lem.~7.3]{Lam_1991_first_course};
	here we need $F$ to be algebraically closed) therefore
	implies the existence of $V_a$ in the proposition.
	It is clear that $S'=\{V_a\where a\in S\}$ has the same cardinality
	as $S$. Moreover, since $F$ is infinite,
	for every $a\in S$, there is $\alpha_a\in F$ such that $a-\alpha_a I\in \GL$.
	Since $a-\alpha_a I$ stabilizes every $U\in S'-\{V_a\}$
	and not $V_a$, it follows that $S'$ is $\GL$-independent.
\end{proof}

Proposition~\ref{PR:indep-num-and-irred-gen} and Theorem~\ref{TH:intro-main}
imply that $\indep[\GL]\Sub(F^n)\geq 2n-1$ for $n>1$ and $F$ algebraically closed.
We do not know if equality
holds in general, but we will show it   for $n\in\{2,3\}$ and any field $F$ with $|F|>2$.
In fact, 
we shall give a complete description of the largest
$\GL$-independent subsets of $\Sub(F^n)$ in these cases.
This, in turn, will be used to give a description of  the largest irredundant generating sets
for $\M[2]$ and $\M[3]$ when $F$ is algebraically closed.
We    treat the fairly simple case $n=2$  here  and the more involved case
$n=3$ in the next section.

\medskip

Recall that $\Sub(F^2)-\{F^2,0\}$ is just the projective line $\bbP^1$ over $F$
and the action $\GL[2]$ on $\bbP^1$ factors via $\nPGL{F}{2}:=\GL[2]/F^\times$,
giving the rise to the standard action of $\nPGL{F}{2}$ on $\bbP^1$ via M\"obius transformations.
This action 
is well-known to be sharply $3$-transitive,
meaning that if $U_1,U_2,U_3,V_1,V_2,V_3\in \bbP^1$
are such that $U_1,U_2,U_3$
are distinct and $V_1,V_2,V_3$
are distinct, then there exists a unique $a\in\nPGL{F}{2}$ satisfying $a(U_i)=V_i$
for all $i$.
Since $|\bbP^1|>3$ when $|F|>2$, this readily implies:

\begin{prp}\label{PR:GL-indep-2}
	When $|F|>2$, the maximal $\GL[2]$-independent subsets of $\Sub(F^2)$
	are those consisting of three $1$-dimensional subspaces. Hence, 
	$\indep[{\GL[2]}]\Sub(F^2)=
	\indep[{\GL[2]}]\bbP^1=3$.
\end{prp}

We now use Propositions~\ref{PR:indep-num-and-irred-gen} 
and~\ref{PR:GL-indep-2} to show
that   up to some natural operations, there is only
one $3$-element irredundant generating set for $\M[2]$ --- the one
exhibited in \eqref{EQ:irred-gen-ex}  with $n=2$.

\begin{prp}\label{PR:max-gen-2}
	Suppose that $F$ is algebraically closed. 
	Let $a_1,a_2,a_3\in \M[2]$ be matrices forming 
	an irredundant generating set for $\M[2]$.
	Then the tuple
	$(a_1,a_2,a_3)$ can be transformed into
	$([\begin{smallmatrix} 1 & 1 \\ 0 & 0 \end{smallmatrix}],
	[\begin{smallmatrix} 0 & 0 \\ 1 & 1 \end{smallmatrix}],
	[\begin{smallmatrix} 1 & 0 \\ 0 & 0 \end{smallmatrix}])$
	(cf.\ \eqref{EQ:irred-gen-ex})
	using the following operations:
	\begin{enumerate}[label=(\arabic*)]
		\item  conjugating $a_1,a_2,a_3$ by some $g\in \GL[2]$;
		\item  replacing one of the $a_i$
	with $\alpha a_i+\beta I$ for some $\alpha\in \units{F}$, $\beta\in F$.
	\end{enumerate}
\end{prp}

Notice that reordering $a_1,a_2,a_3$ is not necessary.

\begin{proof}
	As observed in Proposition~\ref{PR:indep-num-and-irred-gen},
	for $i=1,2,3$, there is $V_i\in \Sub(F^2)$
	which is not invariant under $a_i$ and is invariant under
	$a_j$ for every $j\neq i$. Moreover, $\{V_1,V_2,V_3\}$ is a $\GL[2]$-independent
	set. As we noted earlier, there is $g\in \GL[2]$
	such that $g(V_1)=F[\begin{smallmatrix} 0   \\ 1   \end{smallmatrix}]$, 
	$g(V_2)=F[\begin{smallmatrix} 1   \\   0 \end{smallmatrix}]$ and 
	$g(V_3)=F[\begin{smallmatrix} 1   \\ -1 \end{smallmatrix}]$.
	By replacing $(a_i)_{i=1,2,3}$ with $\{ga_ig^{-1}\}_{i=1,2,3}$,
	we may assume that $V_1=F[\begin{smallmatrix} 0   \\ 1   \end{smallmatrix}]$, $V_2=F[\begin{smallmatrix} 1   \\ 0   \end{smallmatrix}]$ and $V_3=F[\begin{smallmatrix} 1   \\ -1   \end{smallmatrix}]$.
	
	Now, $V_1$ and $V_2$ are eigenspaces of $a_3$; let $\alpha$ and $\beta$ be 
	their respective eigenvalues. We cannot have $\alpha=\beta$ since otherwise $a_3$
	would fix $V_3$. Replacing $a_3$ with $({\beta-\alpha})^{-1}(a_3-\alpha I)$,
	we may assume that $\alpha=0$ and $\beta=1$, i.e., $a_3=0_{V_1}\oplus \id_{V_2}=[\begin{smallmatrix} 1 & 0 \\ 0 & 0 \end{smallmatrix}]$.
	Similarly, we may assume that $a_1=\id_{V_2}\oplus 0_{V_3}=[\begin{smallmatrix} 1 & 1 \\ 0 & 0 \end{smallmatrix}]$
	and 
	$a_2=\id_{V_1}\oplus 0_{V_3}=[\begin{smallmatrix} 0 & 0 \\ 1 & 1 \end{smallmatrix}]$.
	This completes the proof.
\end{proof}

\section{The Case of $3\times 3$ Matrices; Proof of Theorem~\ref{TH:five-elem-irred-classification}}
\label{sec:three-by-three}

In this section, we prove that $\indep[{\GL[3]}]\Sub(F^3)=5$,
classify the largest $\GL[3]$-independent subsets of $\Sub(F^3)$ and  as a consequence, prove
Theorem~\ref{TH:five-elem-irred-classification}.

In what follows, a line (resp.\ plane) in $F^n$
is a $1$- (resp.\ $2$-)dimensional subspace of $F^n$.
We endow $F^n$ with the standard
symmetric bilinear pairing $\Trings{x,y}=\sum_{i=1}^n x_i y_i$
and write $V^\perp$ for $\{x\in F^n\suchthat \Trings{x,V}=0\}$.

\begin{thm}\label{TH:GL-indep-3}
	Suppose that $|F|>2$. Then:
	\begin{enumerate}[label=(\roman*)]
		\item
		$\indep[{\GL[3]}](\Sub(F^3))=5$. In particular,
		every $\GL[3]$-independent subset of $\Sub(F^3)$
		has   $5$ or less elements.
		\item A $5$-element subset $S $
		of $\Sub(F^3)$ is $\GL[3]$-independent if and only if one of the following holds:
		\begin{enumerate}[label=(\arabic*)]
			\item $S$ consists of $3$ lines
			$L_1,L_2,L_3$ and $2$ planes $P_1,P_2$
			such that $L_1\subseteq P_1$, $L_2\subseteq P_2$,
			$L_1+L_2+L_3=F^3$ and $L_i\nsubseteq P_j$ for $i\neq j$;
			\item $S$ consists of $3$ planes
			$P_1,P_2,P_3$ and $2$ lines $L_1,L_2$
			such that $L_1\subseteq P_1$, $L_2\subseteq P_2$,
			$P_1\cap P_2\cap P_3=0$ and $L_i\nsubseteq P_j$ for $i\neq j$.
		\end{enumerate}
	\end{enumerate}
\end{thm}

In order to prove the theorem, we first establish some lemmas.

\begin{lem}\label{LM:no-sum}
	Let $S\subseteq \Sub(F^n)$ be a $\GL$-independent subset.
	Then no $V\in S$ may be obtained from the subspaces in $S-\{V\}$
	using the operations $+$ and $\cap$.
\end{lem}

\begin{proof}
	If that were not the case, then every $a\in\GL$
	stabilizing all subspaces in $S-\{V\}$ would also stabilize $V$, contradicting our assumption
	that $S$ is $\GL$-independent.
\end{proof}

\begin{lem}\label{LM:three-lines-fixed}
	Let $V\subseteq F^n$ be a plane and let $a\in \GL$.
	If $a$ stabilizes three different  lines in $V$, then $a|_V=\alpha \id_V$
	for some $\alpha\in F$. In particular, $a$ stabilizes every line in $V$. 
\end{lem}

\begin{proof}
	Once identifying $V$ with $F^2$, this follows from the fact
	that the action of $\nPGL{F}{2}$ on $\bbP^1$ is sharply $3$-transitive; see
	Section~\ref{sec:GL-indep}.
\end{proof}

\begin{lem}\label{LM:four-lines-fixed}
	Let $a\in \M[3]$.
	If $a$ stabilizes $4$ different lines in $F^3$, then either $3$ of them are contained in a common plane,
	or $a$ is scalar.
\end{lem}

\begin{proof}
	Let $v_1,v_2,v_3,v_4$ be vectors spanning the $4$ lines
	which are invariant under $a$,
	and let $\alpha,\beta,\gamma,\delta\in F$ be their corresponding eigenvalues.
	We may assume without loss of
generality that $v_{1},v_{2},v_{3}$ span $F^{3}$ and so $v_{4}$ may
be represented as $v_{4}=xv_{1}+yv_{2}+zv_{3}$
for some $x,y,z\in F$.
Then 
\[x\delta v_1+y\delta v_2+z\delta v_3 =
\delta v_4=av_{4}=a(xv_{1}+y v_{2}+zv_{3})=x\alpha v_{1}+y \beta  v_{2}+z\gamma  v_{3},\]
hence 
$((\delta=\alpha)\lor(x=0))\land((\delta=\beta)\lor(y=0))\land((\delta=\gamma)\lor(z=0))$.
If $\delta=\alpha=\beta=\gamma$, then $a$ is a scalar matrix.
On the other hand, if
$x=0$ or $y=0$ or $z=0$, then $v_{4}$ is a linear combination
of two   of the three vectors $v_1,v_2,v_3$.
Therefore, three of the lines spanned
by  $v_1,v_2,v_3,v_4$ are contained in a plane.
\end{proof}

\begin{lem}\label{LM:perp-inv}
	Let $V\subseteq F^n$ be a subspace and $a\in \M$.
	Then $a(V)\subseteq V$ if and only if $a^\trans(V^\perp)\subseteq V^\perp$.
\end{lem}

\begin{proof}
	Suppose that $a(V)\subseteq V$ and let $u\in V^\perp$.
	Then for every $v\in V$, we have $\Trings{a^\trans u,v}=\Trings{u,av}\in \Trings{u,V}=0$,
	so $a^\trans u\in V^\perp$. This proves that $a^\trans(V^\perp)\subseteq V^\perp$.
	The converse follows by symmetry.
\end{proof}

\begin{cor}\label{CR:dim-flip}
	Let $S\subseteq \Sub(F^n)$. Then $S$ is $\GL$-independent if and only if
	$S^\perp:=\{V^\perp\where V\in S\}$ is $\GL$-independent.
\end{cor}

\begin{proof}
	Suppose that $S$ is $\GL$-independent. Then for every $V\in S$,
	there is $a_V\in \GL$ such that $V$ is not $a_V$-invariant, but
	every $U\in S-\{V\}$ is $a_V$-invariant.
	By Lemma~\ref{LM:perp-inv}, $V^\perp$ is not $a_V^\trans$-invariant
	while $U^\perp$ is $a_V^\perp$-invariant for every $U\in S-\{V\}$.
	This means that $S^\perp$ is $\GL$-independent.
	The other direction follows by symmetry.
\end{proof}

We    now prove some variants of Theorem~\ref{TH:GL-indep-3}(i) in which
the dimensions of the subspaces in the $\GL[3]$-independent subset of $\Sub(F^3)$ are restricted.
They may be interesting in their own right.

\begin{thm}\label{TH:all-lines}
	Let $S$ be a $\GL[3]$-independent subset of $\Sub(F^3)$
	consisting only of lines (resp.\ planes).
	Then $|S|\leq 4$.
\end{thm}

\begin{proof}
	By Corollary~\ref{CR:dim-flip}, we map replace $S$ with $S^\perp$, so it is enough to address the case where $S$ consists of lines.
	For the sake of contradiction, suppose that $|S|>4$  and let
	$L_1,\dots,L_5$ be $5$ distinct lines in $S$.
	Then for every $i\in\{1,\dots,5\}$, there is $a_i\in \GL[3]$
	stabilizing $L_j$ for each $j\neq i$ and not stabilizing $L_i$.
	Clearly, $a_i$ is not scalar.

	The matrix $a_{5}$ stabilizes the lines $L_{1},L_{2},L_{3},L_{4}$.
	By Lemma~\ref{LM:four-lines-fixed}, three of
them are contained in a plane, and after renumbering,
we may assume that   those
are $L_{1},L_{2},L_{3}$. Now look at $a_{3}$, which stabilizes $L_{1},L_{2},L_{4},L_{5}$.
Applying Lemma~\ref{LM:four-lines-fixed} again,
we see that three of these lines  are contained in a plane.
We now break into cases.

{\it Case I.  $L_{1},L_{2},L_{i}$  are contained in a plane for some $i\in\{4,5\}$.}
Then $L_{1},L_{2},L_{3},L_{i}$ are contained in $L_{1}+L_{2}$.
The matrix $a_3$ stabilizes $L_1$, $L_2$ and $L_i$,
so by  Lemma~\ref{LM:three-lines-fixed}, it must also stabilize $L_3$,
which is a contradiction. 

{\it Case II. $L_{i},L_{4},L_{5}$
are contained
in a plane for some
$i\in\{1,2\}$.} Let $j$ be the member of $\{1,2\}$ different
from $i$. 
Then $L_{i}\subseteq L_{j}+L_{3}$ and $L_i\subseteq L_{4}+ L_{5}$.
Furthermore,   $a_{i}$ stabilizes $L_{j},L_{3},L_{4},L_{5}$,
so by Lemma~\ref{LM:four-lines-fixed}, three of these lines are contained in a plane. 
Since every choice of three lines from $L_{j},L_{3},L_{4},L_{5}$
will include either $L_{j},L_{3}$ or $L_{4},L_{5}$, either
$L_j+L_3$ contains $L_{i}$ and one of $L_4,L_5$,
or $L_4+L_5$
contains $L_i$ and one of $L_j,L_3$.
It follows that four of the lines $L_1,\dots,L_5$ are contained in a plane.
Renumbering, we may assume that  $L_{1},L_{2},L_{3},L_{4}$
are contained in a plane and reach a contradiction as in Case I.
\end{proof}

\begin{remark}\label{RM:four-lines-indep}
	Suppose $|F|>2$.
	With a little more work, it is not hard
	to show that a set of $4$ lines in $F^3$
	is $\GL[3]$-independent if and only if the lines are not contained
	in a plane. 	
	Indeed, if $L_1,L_2,L_3,L_4$ are four lines in $F^3$
	not contained in a plane, then after renumbering,
	we have  $L_1+L_2+L_3=F^3$. Since $(L_1+L_2)\cap (L_2+L_3)\cap (L_3+L_1)=\emptyset$,
	there are distinct $i,j\in\{1,2,3\}$ such that $L_4\nsubseteq L_i+L_j$.
	Renumbering again, we may assume $L_4\nsubseteq L_2+L_3$.
	Now, one can check directly that $\{L_1,\dots,L_4\}$ is $\GL[3]$-independent, e.g.,
	for any $\alpha\in F-\{0,1\}$, the matrix
	$a_4:=\id_{L_1}\oplus \alpha \id_{L_2\oplus L_3}$  
	stabilizes $L_1,L_2,L_3$
	and not $L_4$. 
\end{remark}

\begin{thm}\label{TH:all-but-one-lines}
	Let $S$ be a $\GL[3]$-independent subset of $\Sub(F^3)$
	such that all subspaces in $S$
	except maybe one   have the same dimension.
	Then $|S|\leq 4$.
\end{thm}

\begin{proof}
	By Corollary~\ref{CR:dim-flip}, we may replace $S$ with $S^\perp$ 
	and hence assume that all members of $S$
	are lines except maybe one, which is a plane. The case
	where $S$ consists entirely of lines was treated in Theorem~\ref{TH:all-lines},
	so we may assume that $S$ contains a plane $P_1$.
	For the sake of contradiction, suppose that $S$ contains $4$ additional lines
	$L_1,\dots,L_4$.
	We split into cases.
	
	{\it Case I. None of the lines $L_{1},L_{2},L_{3},L_{4}$ is contained in $P_1$.}
	There is a matrix that stabilizes $L_{1},L_{2},L_{3},L_{4}$
and not $P_1$. By Lemma~\ref{LM:four-lines-fixed}, three of $L_{1},L_{2},L_{3},L_{4}$ are contained in a plane $P_2$, and without loss of generality, we may  assume that those are $L_{1},L_{2},L_{3}$.
In particular, $P_2=L_2+L_3$.
There is $a_1\in \GL[3] $ that stabilizes $L_{2},L_{3},L_{4},P_{1}$ and not
$L_{1}$. Let   $L_{5}=P_{1}\cap P_{2}$. We have $P_1\neq L_2+L_3=P_2$
by Lemma~\ref{LM:no-sum}, so $L_5$ is a line.
We also have $L_5\neq L_2,L_3$, because $L_2,L_3$ are not contained in $P_1$.
Thus, $a_1$ stabilizes three lines in $P_2$, namely,  $L_{2},L_{3},L_{5}$.
By Lemma~\ref{LM:three-lines-fixed}, $a_1$ stabilizes every line in $P_{2}$.
But this means that $a_1$ stabilizes $L_1$, a contradiction.

	{\it Case II. One of   $L_{1},L_{2},L_{3},L_{4}$ is contained in $P_{1}$
and the other three are contained in a   plane. }
	We may assume without
loss of generality that $L_{1}\subseteq P_{1}$, and $P_2:=L_2+L_3+L_4$ is a plane.
If $L_{i}\subseteq P_{1}$ for some $i\in\{2,3,4\}$,
then we would have $P_1=L_1+L_i$, which is impossible by Lemma~\ref{LM:no-sum}.
Thus, $L_2,L_3,L_4\nsubseteq P_1$.
There is  $a_4\in \GL[3]$ stabilizing
$P_{1},L_{1},L_{2},L_{3}$ and not stabilizing $L_{4}$;
it also stabilizes $P_2=L_2+L_3$.
We  have $P_2\neq P_1$ by Lemma~\ref{LM:no-sum}, so
$L_5:=P_1\cap P_2$ is a line stabilized by $a_4$. 
Note also that $L_5\neq L_2,L_3$ because $L_2,L_3\nsubseteq P_1$
while $L_5\subseteq P_1$.
It follows that $a_4$ stabilizes $3$ lines in $P_2$, namely, $L_2,L_3,L_5$.
Thus, by Lemma~\ref{LM:three-lines-fixed}, $a_4$ stabilizes all lines in $P_2$
including $L_4$, but that is a contradiction.

{\it Case III. One of the lines $L_{1},L_{2},L_{3},L_{4}$ is contained in $P_{1}$
and the other three span $F^3$.}
Again, without loss of generality, we may assume that $L_1\subseteq P_1$
and $L_2+L_3+L_4=F^3$.
There is $a_1\in \GL[3]$ stablizing $L_{2},L_{3},L_{4},P_{1}$ and not $L_1$. 
For every $2\leq i<j\leq 4$, let
$L_{ij}=(L_i+L_j)\cap P_1$. Then $a_1(L_{ij})=L_{ij}$.
As in Case II, $L_{i}\nsubseteq P_{1}$
for $i=2,3,4$,  so $L_{i}\ne L_{jk}$ for $2\leq j<k\leq 4$. 
As a result, $L_j+L_{k}=L_j+L_{jk}=L_k+L_{kj}$. 
Now, if $L_{23}= L_{24}$,
then 
we would have $L_2+L_3=L_2+L_{23}=L_2+L_{24}=L_2+L_4$,
contradicting our assumption $L_2+L_3+L_4=F^3$.
We must therefore have $L_{23}\neq L_{24}$,
and similarly $L_{23}\neq L_{34}$ and $L_{24}\neq L_{34}$.
Now, $L_{23}, L_{34}, L_{34}$
are distinct lines in $P_1$ that are stabilized under $a_1$.
By Lemma~\ref{LM:three-lines-fixed}, $a_1$
stabilizes every line in $P_1$, and in particular $a_1(L_1)=L_1$.
This contradicts the choice of $a_1$.
\end{proof}

We can now prove Theorem~\ref{TH:GL-indep-3}. We first prove (ii) and then (i).

\begin{proof}[Proof of the ``only if'' part of Theorem~\ref{TH:GL-indep-3}(ii).]
Let $S$ be $\GL[3]$-independent  $5$-elem\-ent subset of $\Sub(F^3)$.
We need to prove that one of the conditions (1), (2) from Theorem~\ref{TH:GL-indep-3}(ii)
holds. Note that $S$ cannot include $0$ and $F^3$, and by Theorem~\ref{TH:all-but-one-lines},
it includes at least $2$ lines and at least $2$ planes. 
Since $S^\perp$ is also $\GL[3]$-independent (Lemma~\ref{LM:perp-inv}),
and since (1) holds for $S$ if and only (2) holds for $S^\perp$, we may
replace $S$ with $S^\perp$ if needed and assume that $S$ consists of $3$ lines and $2$ planes.
Denote the lines by $L_1,L_2,L_3$ and the planes by $P_1,P_2$.
We will prove   (1) in a series of claims.

{\it Claim 1.  No  $L_i$ ($i=1,2,3$) is contained in both $P_1$ and $P_2$,
and each $P_i$ ($i=1,2$) contains at most one of $L_1,L_2,L_3$.} 
Indeed, if that were not the case, then we would have $L_i=P_1\cap P_2$,
or $P_i=L_j+L_k$ for some $i,k,j$, but this is impossible by Lemma~\ref{LM:no-sum}. 

{\it Claim 2. $L_1+L_2+L_3=F^3$.}
For the sake contradiction, suppose otherwise, i.e.,
$L_1+L_2=L_1+L_3=L_2+L_3$.
There is $a_3\in \GL[3]$ stabilizing $L_1,L_2,P_1,P_2$ and not $L_3$.
Then $a_3$ also stabilizes $L_4:=(L_1+L_2)\cap P_1$.
By Lemma~\ref{LM:no-sum}, $L_1\neq (L_2+L_3)\cap P_1=L_4$
and similarly $L_2\neq L_4$. Thus, $a_3$ stabilizes $L_1,L_2,L_4$, which are all
contained in the plane $L_1+L_2$. By Lemma~\ref{LM:three-lines-fixed},
$a_3$ also stabilizes $L_3$, which is a contradiction.

{\it Claim 3. Each of $P_1,P_2$ contains one of $L_1,L_2,L_3$.}
Again, suppose that this is not the case. 
Then one of $P_1,P_2$ --- say it is $P_1$ ---
does not contain $L_1,L_2,L_3$.
There is $a\in \GL[3]$
stabilizing $L_1,L_2,L_3,P_1$ and not $P_2$.
For every $1\leq i<j\leq 3$, let $P_{ij}=L_i+L_j$ and $L_{ij}=P_{ij}\cap P_1$.
We have $  P_{ij}\neq P_1,P_2$ by Claim~1, so $L_{ij}$
is a line. In addition, $L_{ij}\neq L_i,L_j$
because $L_i,L_j\nsubseteq P_1$.
Thus, $a$ stabilizes three lines in $P_{ij}$, namely $L_i,L_j,L_{ij}$,
and so $a$ stabilizes every line in $P_{ij}$ (Lemma~\ref{LM:three-lines-fixed}).
In particular, $a$ stabilizes $P_{ij}\cap P_2$.
We cannot have $P_{12}\cap P_2=P_{13}\cap P_2=P_{23}\cap P_2$,
because   the intersection of these lines is contained in $P_{12}\cap P_{23}\cap P_{13}$,
which is $0$ by Claim 2. Thus, two of $P_{12}\cap P_2,P_{13}\cap P_2,P_{23}\cap P_2$
are distinct, implying that $P_2=\sum_{i<j}(P_{ij}\cap P_2)$.
But this means that $P_2$ is stabilized under $a$, a contradiction.

{\it Conclusion.}
Thanks to Claim 3, after renumbering, we may assume that $L_1\subseteq P_1$ and $L_i\subseteq P_2$
for some $i\in\{1,2,3\}$.
Since $L_1\nsubseteq P_2$ by Claim 1,   after further renumbering, we may assume that $L_2\subseteq P_2$.
Claim~1 then implies that $L_i\nsubseteq P_j$ unless $i= j$. By Claim 2, $L_1+L_2+L_3=F^3$,
so we established condition (1) of Theorem~\ref{TH:GL-indep-3}(ii).
\end{proof}

\begin{proof}[Proof of the ``if'' part of Theorem~\ref{TH:GL-indep-3}(ii)]
	Suppose first that (1) holds.
	We need to show that $S:=\{L_1,L_2,L_3,P_1,P_2\}$ is $\GL[3]$-independent.
	Put $L_0=P_1\cap P_2$, and
	for   $i=0,1,2,3$, let $v_i$ be a nonzero vector in $L_i$.
	Then $L_0\neq L_1$ by Lemma~\ref{LM:no-sum},
	and so $L_0+L_1=P_1$. Similarly, $L_0+L_2=P_2$,
	and therefore $L_0+L_1+L_2=P_1+P_2=F^3$ (because $P_1\neq P_2$).
	It follows that
	$v_0,v_1,v_2$ is a basis for $F^3$. Write $v_3=\alpha_0 v_0+\alpha_1 v_1+\alpha_2 v_2$
	with $\alpha_0,\alpha_1,\alpha_2\in F$.
	We must have $\alpha_0\neq 0$, because otherwise $F^3=L_1+L_2+L_3\subseteq L_1+L_2$,
	which is absurd.
	We must also have $\alpha_1\neq 0$ since otherwise, $L_3\subseteq L_0+L_2=P_2$, which
	contradicts the assumption $L_3\nsubseteq P_2$. Similarly, $\alpha_2\neq 0$.
	Now, replacing $v_0,v_1,v_2$ with $\alpha_0v_0,\alpha_1v_1,\alpha_2v_2$,
	we may assume that $v_3=v_0+v_1+v_2$.
	Finally, by applying a suitable $g\in \GL[3]$
	to all members of $S$, we may assume that $v_1,v_2,v_0$ are the standard
	basis $e_1,e_2,e_3$ of $F^3$. As a result,
	$L_1=Fe_1$, $L_2=Fe_2$, $L_3=F\left[\begin{smallmatrix} 1 \\ 1 \\ 1\end{smallmatrix}\right]$,
	$P_1=Fe_1+Fe_3$ and $P_2=Fe_2+Fe_3$.
	
	Let $ \alpha \in F-\{0,1\}$  (here we need $|F|>2$) and define
	\begin{align*}
	&a_1=\left[\begin{matrix}
	\alpha \\
	& \alpha \\
	\alpha-1 & & 1
	\end{matrix}\right],
	\quad
	a_2=\left[\begin{matrix}
	\alpha \\
	& \alpha \\
	& \alpha-1 & 1
	\end{matrix}\right],
	\quad
	a_3=\left[\begin{matrix}
	\alpha \\
	& 1 \\
	&  & 1
	\end{matrix}\right],
	\\
	&	
	a_4=\left[\begin{matrix}
	\alpha \\
	& 1 & \alpha-1 \\
	&  & \alpha
	\end{matrix}\right],
	\quad
	a_5=\left[\begin{matrix}
	1 & & \alpha-1 \\
	& \alpha \\
	&  & \alpha
	\end{matrix}\right].
	\end{align*}
	Putting $L_4:=P_1$ and $L_5:=P_2$,
	it is routine to check that each $a_i$ stabilizes every member of $S-\{L_i\}$
	and not $L_i$. Thus, $S$ is $\GL[3]$-independent.
	
	Suppose now that (2) holds. Then $S^\perp$ (notation as in Corollary~\ref{CR:dim-flip})
	satisfies  (1) and is therefore $\GL[3]$-independent.
	By Corollary~\ref{CR:dim-flip}, this means that $S=(S^\perp)^\perp$ is also
	$\GL[3]$-independent.
\end{proof}

\begin{proof}[Proof of Theorem~\ref{TH:GL-indep-3}(i)]
	For the sake of contradiction, suppose that $S$ is a $6$-element
	subset of $\Sub(F^3)$ which is $\GL[3]$-independent.
	Then every $5$-element subset of $S$ is also $\GL[3]$-independent,
	meaning that  Theorem~\ref{TH:GL-indep-3}(ii), which we already proved,
	applies to it. In particular, every $5$-element subset of $S$
	consists of $3$ lines and $2$ planes, or $2$ lines and $3$ planes.
	As a result, $S$ must consist of $3$ lines, denoted $L_1,L_2,L_3$,
	and $3$ planes, denoted $P_1,P_2,P_3$.
	Looking at the subset $\{L_1,L_2,L_3,P_1,P_2\}$ of $S$,
	Theorem~\ref{TH:GL-indep-3}(ii)   tells us that, after renumbering,
	we may assume that $L_1\subseteq P_1$, $L_2\subseteq P_2$ and $L_3\nsubseteq P_1,P_2$.
	Since Theorem~\ref{TH:GL-indep-3}(ii) also applies to the subset
	$\{L_1,L_3,P_1,P_2,P_3\}$, we must also have $L_3\subseteq P_3$
	and $P_1\cap P_2\cap P_3=0$.
	This means that $L_i\subseteq P_i$ for $i=1,2,3$,
	and as in Claim~1 in the proof of Theorem~\ref{TH:GL-indep-3}(ii),
	$L_i\nsubseteq P_j$ whenever $i\neq j$.

For $1\leq i<j\leq 3$,
let $L_{ij}=P_i\cap P_j$, and let
$L_{0}=(L_{1}+L_{2})\cap P_{3}$. The subspace $L_0$ is a line
because $L_1+L_2\neq P_3$ (Lemma~\ref{LM:no-sum}). 
We now break into cases.

{\it Case I. $L_{13}=L_{0}$ or $L_{23}=L_0$.} 
By replacing $L_1$ with $L_2$
and $P_1$ with $P_2$,
it is enough to consider the case $L_{13}=L_0$.
Then $L_0$ is a subspace of $P_1$, $P_3$ and $(L_1+L_2)$.
Since
$L_1+L_2\neq P_1$ (Lemma~\ref{LM:no-sum}), this means that  
$ L_{0}=(L_{1}+L_{2})\cap P_{1}=L_1$. It follows that $L_1=L_0\subseteq P_3$,
which contradicts our earlier conclusions.

{\it Case II. $L_0\neq L_{13},L_{23}$.}
There is $a\in \GL[3]$ stabilizing 
$L_{1},L_{2}, P_{1},P_{2},P_3$ and not $L_3$.
Then $a$ stabilizes 
$L_{13},L_{23},L_{0}$.
These three lines are contained in $P_3$, and we have $L_0\neq L_{13},L_{23}$ by assumption, 
and $L_{13}\neq L_{23}$ because $L_{13}\cap L_{23}=P_1\cap P_2\cap P_3=0$.
Thus, by Lemma~\ref{LM:three-lines-fixed}, $a$
stabilizes every line $P_3$  including $L_3$. This contradicts the choice of $a$.
\end{proof}

\begin{remark}
	Regarding the assumption $|F|>2$ in Theorem~\ref{TH:GL-indep-3},
	the inequality $\indep[{\GL[3]}]\Sub(F^3)\leq 5$ and
	the   ``only if'' part of (ii) hold even when $|F|=2$;
	this is evident from the proofs.
	In fact, one can remove the assumption $|F|>2$
	entirely by replacing the group $\GL[3]$
	with the multiplicative monoid of $\M[3]$,
	and declaring   a subset $S$ of $\Sub(F^3)$
	to be a $\M[3]$-independent if for every
	$V\in S$, there $a\in\M[3]$ such that $V$
	is not $a$-invariant, while very member of $S-\{V\}$
	is $a$-invariant.
\end{remark}

We  turn to prove Theorem~\ref{TH:five-elem-irred-classification}.
We first show:

\begin{lem}\label{LM:many-inv-subspaces}
Let $v_1,\dots,v_4\in F^3$ be vectors such that every
$3$ of them span $F^3$.
Suppose that $v_1,v_2,v_3$
are eigenvectors of a matrix $a\in\M[3]$ with corresponding eigenvalues
$\alpha_1,\alpha_2,\alpha_3$.
If $Fv_1+Fv_4$ is $a$-invariant, then $\alpha_2=\alpha_3$.
\end{lem}

\begin{proof}
There are $\beta_1,\beta_2,\beta_3\in F$
such that $v_4=\beta_1 v_1+\beta_2 v_2+\beta_3 v_3$.
Our assumptions on $v_1,\dots,v_4$
imply that $\beta_1,\beta_2,\beta_3\neq 0$.
Replacing $v_i$ with $\beta_i v_i$ for $i=1,2,3$,
we may assume that $v_4=v_1+v_2+v_3$.
Since $Fv_1+Fv_4$ is $a$-invariant,
\[\alpha_2 v_2+\alpha_3 v_2 = a(v_2+v_3)=a(v_4-v_1)\in Fv_1+Fv_4=Fv_1+F(v_2+v_3).\] 
Since $v_1,v_2,v_3$
is a basis to $F^3$, this 
is possible only if $\alpha_2=\alpha_3$.
\end{proof}

\begin{proof}[Proof of Theorem~\ref{TH:five-elem-irred-classification}]
Recall that for $S,T\subseteq \M$, we write $S\sim T$
if $S$ can be transformed into $T$ using the operations
(1)--(3) from the Introduction, i.e., using
(1) simultaneous conjugation, (2) simultaneous transposition 
and (3) replacing a matrix $a$ in the set with $\alpha a+\beta I$ for some $\alpha\in\units{F}$ and $\beta\in F$.

Let $S\subseteq \M[3]$. 
Suppose first that $S\sim S_\alpha$ for some $\alpha\in F$ ($S_\alpha$ is as in the theorem).
We need to show that $S_\alpha$, and hence $S$,
is an irredundant generating set.
Let $a_1,a_2,a'_1,a'_2,b$
be as in \eqref{EQ:irred-gen-ex} with $n=3$.
Then $a_1,a_2,a'_1,a'_2\in S_\alpha$ and
\[
b = \left[
\begin{smallmatrix}
1 & 1 & 0 \\
0 & 0 & 0 \\
0 & 0 & 0 
\end{smallmatrix}
\right]
\left[
\begin{smallmatrix}
1 & 0 & 0 \\
0 & 0 & 0 \\ 
0 & 0 & \alpha
\end{smallmatrix}
\right]\in\Trings{S_\alpha}.
\]
Thus, by Proposition~\ref{PR:irred-example}, $S_\alpha$ generates $\M[3]$.
That no proper subset of $S_\alpha$ generates $\M[3]$ is shown exactly as in the proof
of that proposition.

Suppose now that $S=\{a_1,\dots,a_5\}$ is an irredundant generating set for $\M[3]$. In particular,
none of $a_1,\dots,a_5$ is a scalar matrix.
By Proposition~\ref{PR:indep-num-and-irred-gen}, for every $i\in\{1,\dots,5\}$,
there is a subspace $V_i\subseteq F^3$ that is invariant under all members
of $S-\{a_i\}$ and not invariant under $a_i$. Moreover, $S'=\{V_1,\dots,V_5\}$
is a $5$-element $\GL[3]$-independent subset of $\Sub(F^3)$.
Thus, it satisfies one the conditions  (1), (2) of Theorem~\ref{TH:GL-indep-3}(ii).

If $S'$ satisfies condition (2), then, using Lemma~\ref{LM:perp-inv},
we may replace $S$ with $\{a_1^\trans, \dots,a_5^\trans\}$ (operation (2))
and $S'$ with $S'^\perp$. After this replacement, $S'$ satisfies condition (1) 
of Theorem~\ref{TH:GL-indep-3}. It is therefore enough to consider the case
where $S'$ satisfies (1).

Now, as in the proof of the ``if'' part of Theorem~\ref{TH:GL-indep-3}(ii),
by conjugating all the $a_i$ with a suitable $g\in \GL[3]$ (operation (1))
and replacing $S'$ with $\{gV_1,\dots,gV_5\}$,
we may assume that
$V_1=Fe_1$, $V_2=Fe_2$, $V_3=F\left[\begin{smallmatrix} 1 \\ 1 \\ 1\end{smallmatrix}\right]$,
$V_4=Fe_1+Fe_3$ and $V_5=Fe_2+Fe_3$.
We further let $v_3=e_1+e_2+e_3$ so
that $V_3=Fv_3$.

Observe that:
\begin{itemize}
\item $a_1$ stabilizes $Fe_2 $, $Fv_3 $, $Fe_3=V_4\cap V_5$
and also $V_4= F e_3+F e_1 $.
\item $a_2$ stabilizes $F e_1$, $Fv_3$, $Fe_3=V_4\cap V_5$ and also $V_5= F e_3+Fe_2$.
\item $a_4$ stabilizes $Fe_1$, $Fe_2$, $F v_3$ and $V_5=F e_2+Fe_3$.
\item $a_5$ stabilizes $Fe_1$, $Fe_2$, $F v_3$ and $V_4=F e_1+Fe_3$.
\end{itemize}
Let us consider $a_1$. Then $e_2,e_3,v_3$
are eigenvectors of $a_1$; let $\alpha,\beta,\gamma$
denote the corresponding eigenvalues.
Since $Fe_3+Fe_1$ is also  $a_1$-invariant,
Lemma~\ref{LM:many-inv-subspaces} says that $\alpha=\gamma$,
and
since $a_1$ is not scalar, we  must have $\alpha\neq \beta$.
By replacing $a_1$ with $(\beta-\alpha)^{-1}(a_1-\alpha I)$ (operation (3)),
we may further assume that $\alpha=0$ and $\beta=1$. Thus, 
\[
a_1=
\left[\begin{smallmatrix}
0 & 0 & 0 \\
0 & 0 & 0 \\ 
-1 & 0 & 1
\end{smallmatrix}
\right].
\]
By treating $a_2$, $a_4$, $a_5$ in the same manner, we may further assume that
\[
a_2
=
\left[\begin{smallmatrix}
0 & 0 & 0 \\
0 & 0 & 0 \\ 
0 & -1 & 1
\end{smallmatrix}
\right],
\quad
a_4
=
\left[\begin{smallmatrix}
0 & 0 & 0 \\
0 & 1 & -1 \\ 
0 & 0 & 0
\end{smallmatrix}
\right],
\quad
a_5=
\left[\begin{smallmatrix}
1 & 0 & -1 \\
0 & 0 & 0 \\ 
0 & 0 & 0
\end{smallmatrix}
\right].
\]

The matrix $a_3$ stabilizes $V_1=Fe_1$, $V_2=Fe_2$ and $V_4\cap V_5=Fe_3$,
so $a_3$ is diagonal and non-scalar. Write $a_3=\operatorname{diag}(\alpha,\beta,\gamma)$.
Then we have reduced to the case  
\[
S=\left\{
\left[\begin{smallmatrix}
0 & 0 & 0 \\
0 & 0 & 0 \\ 
0 & -1 & 1
\end{smallmatrix}
\right],
\left[\begin{smallmatrix}
0 & 0 & 0 \\
0 & 1 & -1 \\ 
0 & 0 & 0
\end{smallmatrix}
\right],
\left[\begin{smallmatrix}
1 & 0 & -1 \\
0 & 0 & 0 \\ 
0 & 0 & 0
\end{smallmatrix}
\right],
\left[\begin{smallmatrix}
0 & 0 & 0 \\
0 & 0 & 0 \\ 
-1 & 0 & 1
\end{smallmatrix}
\right],
\left[\begin{smallmatrix}
\alpha & 0 & 0 \\
0 & \beta & 0 \\ 
0 & 0 & \gamma
\end{smallmatrix}
\right]
\right\}.
\]
Now conjugate all members of $S$ by $ \left[\begin{smallmatrix}
1 & 0 & 0 \\
0 & 0 & -1 \\ 
0 & 1 & 0
\end{smallmatrix}
\right]$ to get
to
\[
T:=\left\{
\left[\begin{smallmatrix}
1 & 1 & 0 \\
0 & 0 & 0 \\ 
0 & 0 & 0
\end{smallmatrix}
\right],
\left[\begin{smallmatrix}
0 & 0 & 0 \\
1 & 1 & 0 \\ 
0 & 0 & 0
\end{smallmatrix}
\right],
\left[\begin{smallmatrix}
0 & 0 & 0 \\
0 & 1 & 1 \\ 
0 & 0 & 0
\end{smallmatrix}
\right],
\left[\begin{smallmatrix}
0 & 0 & 0 \\
0 & 0 & 0 \\ 
0 & 1 & 1
\end{smallmatrix}
\right],
\left[\begin{smallmatrix}
\alpha & 0 & 0 \\
0 & \gamma & 0 \\ 
0 & 0 & \beta
\end{smallmatrix}
\right]
\right\}.
\]
We have  $\alpha\neq \gamma$ or $\beta\neq \gamma$.
Since
conjugating $T$ by $ \left[\begin{smallmatrix}
0 & 0 & 1 \\
0 & 1 & 0 \\ 
1 & 0 & 0
\end{smallmatrix}
\right]$
results in swapping $\alpha$ and $\beta$,
we may assume that $\alpha\neq \gamma$.
Now, replacing the matrix
$a:= \operatorname{diag}(\alpha,\gamma,\beta)$ 
in $T$ with
$(\alpha-\gamma)^{-1}(a-\gamma I)$ brings us
to $S_\delta$ for $\delta=\frac{\beta-\gamma}{\alpha-\gamma}$.
This shows that our original $S$ is equivalent to $S_\delta$.
\end{proof}

Since operations (1)--(3) do not change the diagonalizability,
or number of eigenvalues of 
the matrices in the set, we get:

\begin{cor}
	Let $F$ be an algebraically closed field 
	and let $S$ be a $5$-element irredundant generating set for $\M[3]$.
	Then all matices in $S$ are diagonalizable. Moreover,
	all members of $S$ except maybe one have exactly $2$  eigenvalues.
\end{cor}

\begin{remark}\label{RM:no-of-alphas}
	In general, a   $5$-element irredundant generating set for $\M[3]$
	may admit more than one $\alpha\in F$ for which $S\sim S_\alpha$.
	Indeed, for $\alpha\neq 0$, we have $S_\alpha\sim S_{\alpha^{-1}}$,
	because conjugating all members of $S_\alpha$ by
	$ \left[\begin{smallmatrix}
0 & 0 & 1 \\
0 & 1 & 0 \\ 
1 & 0 & 0
\end{smallmatrix}
\right]$
	gives    
	\[
\left \{
\left[
\begin{smallmatrix}
1 & 1 & 0 \\
0 & 0 & 0 \\
0 & 0 & 0 
\end{smallmatrix}
\right],
\left[
\begin{smallmatrix}
0 & 0 & 0 \\
1 & 1 & 0 \\
0 & 0 & 0 
\end{smallmatrix}
\right],
\left[
\begin{smallmatrix}
0 & 0 & 0 \\
0 & 1 & 1 \\
0 & 0 & 0 
\end{smallmatrix}
\right],
\left[
\begin{smallmatrix}
0 & 0 & 0 \\
0 & 0 & 0 \\
0 & 1 & 1 
\end{smallmatrix}
\right],
\left[
\begin{smallmatrix}
\alpha & 0 & 0 \\
0 & 0 & 0 \\ 
0 & 0 & 1
\end{smallmatrix}
\right]
\right\}
\]
and after   multiplying the last matrix by $\alpha^{-1}$ we get  to $S_{\alpha^{-1}}$.

On the other hand, there are at most $6$ values of $\alpha$ for which $S\sim S_\alpha$.
To see this, note first that $S_\alpha$ contains a matrix with $3$ different eigenvalues
if and only if $\alpha\neq 0,1$. Since operations (1)--(3) do not change the number
of eigenvalues of the matrices in the set, $S_\alpha\nsim S_0,S_1$ if $\alpha\neq 0,1$.
Suppose now that $\alpha,\beta\in F-\{0,1\}$ and $S_\alpha\sim S_\beta$.
Since operations (1) and (2) do not affect the eigenvalues of the matrices in the set,
and since $a_\alpha:=\mathrm{diag}(1,0,\alpha)$ (resp.\ $a_\beta$) is the only matrix in $S_\alpha$
(resp.\ $S_\beta$) having $3$ eigenvalues, there are some $x\in F^\times$, $y\in F$
such that $xa_\alpha+yI$ and $a_\beta$ have the same eigenvalues, i.e.,
the sets $\{x+y,y,\alpha x+y\}$ and $\{1,0,\beta\}$ are equal.
In particular, $0,1\in \{x+y,y,\alpha x+y\}$. Elementary linear algebra
shows that there are at most $6$ pairs   $(x,y)\in F^2$ for which this is possible. Since
$\beta$ is determined by $x$, $y$ and $\alpha$, it can take at most    $6$ different values once
$\alpha$ is fixed. 
\end{remark}

\section{An Application to Generation of Azumaya Algebras}
\label{sec:application}

In this final section, we use Theorem~\ref{TH:five-elem-irred-classification} and Proposition~\ref{PR:max-gen-2}
to compute the dimension of the variety of irredundant generating $(2n-1)$-tuples of $n\times n$
matrices for $n\in\{2,3\}$. We then apply this result to slighly improve a result from \cite{First_2022_generators_of_alg_over_comm_ring}
about generation of Azumaya algebras.

\medskip

Let $n,r\in\N$.
We recall some facts and notation from \cite{First_2022_generators_of_alg_over_comm_ring}.
Let $V_r^{(n)}$ denote the affine space over $F$ underlying the $F$-vector space $\M(F)^r$.
One can also think of $V_r^{(n)}$ as the $F$-scheme of $r$-tuples of $n\times n$ 
matrices.
In \cite[\S2]{First_2022_generators_of_alg_over_comm_ring},
the authors show that the $r$-tuples of matrices not generating 
$\M$ underlie a closed subscheme $Z_r^{(n)}$ of $V_r^{(n)}$.
Formally,   $Z_r^{(n)}$ is a closed subsheme of $V_r^{(n)}$
having the property that for every $F$-field $K$,
$Z_r^{(n)}(K)$ is the set of $r$-tuples $(a_1,\dots,a_r)\in\M(K)^r=V_r^{(n)}(K)$ that
do not generate $\M(K)$. The complement  $U_r^{(n)}:=V_r^{(n)}-Z_r^{(n)}$ is the $F$-variety
of $r$-tuples generating $\M$.
As usual, let $\PGL$ be  $F$-algebraic group of algebra automorphisms of $\M$. The evident action of $\PGL$ on $V_r^{(n)}$
restricts to $Z_r^{(n)}$ and $U_r^{(n)}$, where on $U_r^{(n)}$, the action is free
\cite[\S3]{First_2022_generators_of_alg_over_comm_ring}.

Let $p_i:V_r^{(n)}=(V_1^{(n)})^r\to (V_1^{(n)})^{r-1}= V_{r-1}^{(n)}$ denote
the projection omitting the $i$-th coordinate  and define 
\begin{equation}\label{EQ:Ir-dfn}
I_r^{(n)}:=U_r^{(n)} \cap \bigcap_{i=1}^r p_i^{-1}(Z_{r-1}^{(n)}).
\end{equation}
(Here, the inverse image means pullback
and the intersection is scheme-theoretic.)
Then $I_r^{(n)}$ is a closed subscheme of $U_r^{(n)}$. Moreover, it is the  
$F$-scheme of irredundant generating $r$-tuples for $\M$
in the sense that for any $F$-field $K$, its $K$-points
are the  $r$-tuples $(a_1,\dots,a_r)\in\M(K)$ which form an irredundant
generating set. 
(Note that  $(a_1,\dots,a_r)\in I_r^{(n)}(K)$ implies
that the $a_i$ must all be distinct and so $\{a_1,\dots,a_r\}$
has exactly $r$ elements.)

For $n>1$, our Theorem~\ref{TH:intro-main} is equivalent to saying that $I_r^{(n)}\neq\emptyset$
if and only if $1<r\leq 2n-1$, but it does not tell us what is $\dim I_r^{(n)}$. 
We use the finer  Theorem~\ref{TH:five-elem-irred-classification} and Proposition~\ref{PR:max-gen-2}
to determine $\dim I_{2n-1}^{(n)}$ when $n\in\{2,3\}$.

\begin{prp}\label{PR:dim-Ir}
	$\dim I^{(2)}_3=9$
	and  $\dim I_5^{(3)}=19$.
\end{prp}

\begin{proof}
	Since the formation of $Z_r^{(n)}$ and $U_r^{(n)}$ commutes
	with base-change (see \cite[\S2]{First_2022_generators_of_alg_over_comm_ring}),
	and since passing to the algebraic closure does not affect the dimension,
	it is enough to prove the proposition when $F$ is algebraically closed.
	We may further replace $I_r^{(n)}$ with its underlying $F$-variety.
	Now that $F$ is algebraically closed, we will freely identify $F$-varieties
	with their set of $F$-points, endowed with the Zariski topology,
	and define morphisms of $F$-varieties by just specifying them on $F$-points.
	
	We write $\Gm$, resp.\ $\Ga$, for the multiplicative, resp.\ additive, algebraic
	group over $F$. By $\Gm \ltimes \Ga$ we mean the semidirect project
	with respect to the action $(x,y)\mapsto xy:\Gm\times \Ga\to \Ga$.
	In general, given a semidirect product $G\ltimes H$, we will consider
	$G$ and $H$ as subgroups of $G\ltimes H$.

\smallskip

	\emph{We begin with showing that $\dim I_3^{(2)}=9$.}
	The algebraic group $G:=\PGL[2]\times (\Gm \ltimes \Ga)^3$ acts
	on $V_3^{(2)}$ with $\PGL[2]$ acting as above
	and  $(\Gm\ltimes \Ga)^3$ acting via 
	$(\alpha_i,\beta_i)_{i=1,2,3}\cdot (a_i)_{i=1,2,3}=(\alpha_i a_i+\beta_i I)_{i=1,2,3}$.
	It is clear that $I_3^{(2)}$ is stable under $G$,
	and by  
	Proposition~\ref{PR:max-gen-2}, the action of $G$ on $I_3^{(2)}$ is transitive. 
	Thus, for any $s\in I_3^{(2)}$,
	we have $\dim I^{(2)}_3= \dim G-\dim \Stab_{G}(s)$.
	
	We now determine
	the dimension of $H:=\Stab_G(s)$ for   $s=(s_1,s_2,s_3):=([\begin{smallmatrix} 1 & 1 \\ 0 & 0 \end{smallmatrix}],
	[\begin{smallmatrix} 0 & 0 \\ 1 & 1 \end{smallmatrix}],
	[\begin{smallmatrix} 1 & 0 \\ 0 & 0 \end{smallmatrix}])$.
	To begin, note that if $(g,(\alpha_i,\beta_i)_{i=1,2,3})\in \Stab_G(s)$,
	then $\alpha_i s_i+\beta_i$ has the same eigenvalues as 
	$s_i$ for every $i$. Thus, $\{0,1\}=\{\beta_i,\alpha_i+\beta_i\}$,
	meaning that $(\alpha_i,\beta_i)\in\{(1,0),(-1,1)\}$.
	This means that $H\cap \PGL[2]$
	is of finite index in $H$.
	Let $g\in H\cap \PGL[2]$. Then $g$ fixes  $s_1,s_2,s_3$.
	Since $s_1,s_2,s_3$ generate $\M[2]$, this means $g=1_{\PGL[2]}$,
	so $H\cap \PGL[2]=\{1\}$. We conclude that $H$ is finite,
	hence $\dim H=0$.
	
	To finish, $\dim I_3^{(2)}=\dim G-\dim H=3+3\cdot 2-0=9$.
	
\smallskip

	\emph{We turn to show that $\dim I_5^{(3)}=19$.}
	Similarly to  the   case of $I^{(2)}_3$, the algebraic group $(\Gm \ltimes \Ga)^5$ acts
	on $V_5^{(3)}$. In addition, 
	$\PGL[3]$ acts on $V_5^{(3)}$ as usual,   $S_5$ acts on $V_5^{(3)}$
	by permuting the coordinates, and $S_2$
	acts on $V_5^{(3)}$ by letting the nontrivial element acts as 
	$(a_1,\dots,a_5)\mapsto (a_1^\trans,\dots,a_5^\trans)$.
	These actions glue into an action of
 	$G:=(S_2\ltimes\PGL[3])\times(S_5\ltimes (\Gm \ltimes \Ga)^5)$   on $V_5^{(3)}$.
 	This action restricts to $I_5^{(3)}$,
 	and by Theorem~\ref{TH:five-elem-irred-classification},
	every $G$-orbit in $I_5^{(3)}$
	contains a representative of the form
	\[
	s_\alpha:=
\left(	
	\left[
	\begin{smallmatrix}
1 & 1 & 0 \\
0 & 0 & 0 \\
0 & 0 & 0 
\end{smallmatrix}
\right],
\left[
\begin{smallmatrix}
0 & 0 & 0 \\
1 & 1 & 0 \\
0 & 0 & 0 
\end{smallmatrix}
\right],
\left[
\begin{smallmatrix}
0 & 0 & 0 \\
0 & 1 & 1 \\
0 & 0 & 0 
\end{smallmatrix}
\right],
\left[
\begin{smallmatrix}
0 & 0 & 0 \\
0 & 0 & 0 \\
0 & 1 & 1 
\end{smallmatrix}
\right],
\left[
\begin{smallmatrix}
1 & 0 & 0 \\
0 & 0 & 0 \\ 
0 & 0 & \alpha
\end{smallmatrix}
\right]
	\right).
	\]
	Define a morphism $\vphi:  G\times \bbA^1\to I_5^{(3)}$
	by
	$\vphi(g,\alpha)=g(s_\alpha)$.
	Then $\vphi$ is surjective. We have $\dim (G\times \bbA^1)=
	8+5\cdot 2+1=19$. Thus, 
	if we could show that $\vphi$ has finite fibers, then 
	the Fiber Dimension Theorem would imply the desired conclusion
	$\dim I_5^{(3)}=19$.
	
	It remains to show that $\vphi$ has finite fibers.
	Fix some $s\in I_5^{(3)}$.
	By Remark~\ref{RM:no-of-alphas}, there are finitely many
	$\alpha$-s admitting a $g\in G$ with $gs_\alpha=s$; denote them
	by $\alpha_1,\dots,\alpha_t$  and for each $\alpha_i$,
	choose $g_i\in G$ with $g_i s_{\alpha_i}=s$.
	Define 
	\[\psi:\bigsqcup_{i=1}^t  \Stab_{G}(s_{\alpha_i})\to \vphi^{-1}(s)\]
	by setting $\psi( h)= (g_i h,\alpha_i)  $ on $\Stab_{G}(s_{\alpha_i})$.
	Then $\psi$ is a morphism and it is easy to see that it has an inverse
	$\psi':\vphi^{-1}(s)\to \bigsqcup_{i=1}^t  \Stab_{G}(s_{\alpha_i})$
	defined by sending $(x,\alpha_i)$ to $g_i^{-1} x$ in $\Stab_G(s_{\alpha_i})$.
	Thus, $|\vphi^{-1}(s)|=\sum_{i=1}^t
	|\Stab_G(s_{\alpha_i})|$
	and we are reduced into showing that $\Stab_G(s_{\alpha})$ is finite for all
	$\alpha\in F$.

	Let $\alpha\in F$.	
	Since $H:=\PGL[3]\times (\Gm \ltimes \Ga)^5$
	has finite index in $G$, it is enough to show that 
	$\Stab_G(s_{\alpha})\cap H=\Stab_H(s_\alpha)$ is finite.
	Let $(h,(\alpha_i,\beta_i)_{i=1,\dots,5})\in \Stab_H(s_\alpha)$.
	As in the case $n=2$, each $(\alpha_i,\beta_i)$
	can take only finitely many values, so $\Stab_{\PGL[3]}(s_\alpha)=\Stab_H(s_\alpha)\cap \PGL[3]$
	has finite index in $\Stab_H(s_\alpha)$. Finally,
	since the matrices in $s_\alpha$ generate $\M[3]$,
	we have $\Stab_{\PGL[3]}(s_\alpha)=\{1\}$. We conclude
	that $\Stab_G(s_\alpha)$ is finite.
\end{proof}

Let $R$ be a ring and let $B$ be an $R$-algebra that is finitely generated as
an $R$-module. We say that a   tuple of elements $(b_1,\dots,b_r)\in B^r$
generates $B$ \emph{redundantly} if there is $i\in\{1,\dots,r\}$ such that $  b_1,\dots,\hat{b}_i,\dots,b_r  $
($\hat{b}_i$ means omitting $b_i$) generate  $R$. (For example, this holds if the tuple
$(b_1,\dots,b_r)$ generates $B$ and includes repetitions.)
We say the a generating tuple $(b_1,\dots,b_r)$
is \emph{locally redundant} if for every
maximal ideal $\frakm\in \Max R$,
the images $(b_1+\frakm B,\dots,b_r+\frakm B)$
redundantly generate
$B(\frakm):=B/\frakm B\cong B\otimes_R (R/\frakm)$ as an algebra over the field $R/\frakm$.

Suppose now that $R$ is a finitely generated $F$-ring of Krull dimension $d$
and $F$ is infinite.
According to \cite[Thm.~6.1]{First_2022_generators_of_alg_over_comm_ring}, if $d<\dim V^{(n)}_r-\dim Z^{(n)}_r $,
then every degree-$n$ Azumaya algebra over $R$ can be generated by $r$ elements.
By further showing that $\dim Z^{(n)}_r=rn^2-(r-1)(n-1)$,
it was concluded in \cite[Thm.~1.5(a)]{First_2022_generators_of_alg_over_comm_ring}
that every Azumay algebra of degree $n$
over $R$ can be generated by $\floor{\frac{d}{n-1}}+2$
elements.
We now show that knowing $\dim I^{(n)}_r$   leads to analogous  consequences about the existence
of small
\emph{locally redundant} generating tuples for Azumaya algebras over $R$.

\begin{thm}\label{TH:dim-Ir-conseq}
	Let $R$ be a finitely generated ring over an infinite field $F$
	and let $d$ be the Krull-dimension of $R$.
	If $d<\dim V_r^{(n)} - \max\{\dim I^{(n)}_r,\dim Z^{(n)}_r\}$, then every
	degree-$n$ Azumaya algebra over $R$ has a locally redundant generating
	$r$-tuple.
\end{thm}

We first prove:

\begin{lem}\label{LM:group-action-on-union}
	Let $G$ be an $F$-algebraic group  acting on an affine $F$-scheme
	$X=\Spec R$. For $i=1,2$, let $Y_i$ be a closed subscheme of $X$
	corresponding to an ideal $I_i$ of $R$.
	Let $Z$ be the closed subscheme of $X$ determined by $I_1I_2$,
	i.e., $Z=\Spec R/(I_1I_2)$; we have $Z=Y_1\cup Y_2$ as sets. 
	If the action of $G$
	on $X$ restricts to $Y_1$ and $Y_2$, then it also restricts to $Z$.
\end{lem}

\begin{proof}
	Let $\alpha:F[X]\to F[G]\otimes_F F[X] $ be the $F$-algebra
	morphism corresponding to the action   $G\times X\to X$.
	For $i=1,2$, the condition that the action of $G$ restricts to $Y_i$
	is equivalent to having $\alpha(I_i)\subseteq F[G]\otimes_F I_i$.
	When this holds for $i=1,2$, we have
	$\alpha(I_1I_2)\subseteq \alpha(I_1)\alpha(I_2)\subseteq \F[G]\otimes_F (I_1I_2)$,
	so the action of $G$ restricts to $Z$.
\end{proof}

\begin{proof}[Proof of Theorem~\ref{TH:dim-Ir-conseq}]
	For brevity, we omit the superscript $(n)$ in $V_r^{(n)}$, $Z_r^{(n)}$, etc.
	Throughout, all schemes are over $F$ and all morphisms are $F$-morphisms.

	We recall some of the setting of 
	\cite[\S4]{First_2022_generators_of_alg_over_comm_ring} for the proof.
	Write  $G=\PGL[n]$,
	and recall that $G $
	acts on $V_1$. The Azumaya algebra $B$
	corresponds to a (left) $G$-torsor $E\to \Spec R$. 
	Moreover, by \cite[Lem.~4.3]{First_2022_generators_of_alg_over_comm_ring},
	$B$ is isomorphic to the algebra of $G$-equivariant morphisms
	from $E$ to $V_1$, denoted $\Mor_G(E,V_1)$.
	(Here, the addition and product on $\Mor_G(E,V_1)$ are induced by the corresponding operations on $V_1$.)
	We may therefore assume that $B=\Mor_G(E,V_1)$.
	Furthermore, if $ R\to R'$ is a ring homomorphism,
	$B'=B\otimes_R R'$
	and $E'$ is the pullback of $E\to \Spec R$ along $ \Spec R'\to \Spec R$,
	then $E'\to\Spec R'$ is the $G$-torsor corresponding
	to $B'$, and if we identify
	$B'$ with $\Mor_G(E',V_1)$, then the natural map $b\mapsto b\otimes 1_{R'} :B\to B'$
	coincides with the map $\Mor_G(E,V_1)\to \Mor_G(E',V_1)$ given by
	precomposition with the projection $E'\to E$. 
	
	Returning to the proof, a result of our
	assumption $B=\Mor_G(E,V_1)$ is that for every $r\in\N$,
	elements of $\Mor_G(E,V_r)=\Mor_G(E,V_1)^r$
	are the same as $r$-tuples of elements from $B$.
	Let $\vphi:E\to V_r$ be a $G$-equivariant morphism, and let $(b_1,\dots,b_r)\in B^r$
	be its corresponding $r$-tuple.
	By the Claim in the proof of  \cite[Prop.~4.3]{First_2022_generators_of_alg_over_comm_ring},
	$b_1,\dots,b_r$ generate $B$ if and only if $\im(\vphi)\subseteq U_r=V_r-Z_r$.
	We  make   a similar claim for $U_r-I_r=V_r-(Z_r\cup I_r)$.	
	
\smallskip	
	
	{\it Claim.} Consider an $r$-tuple $(b_1,\dots,b_r)\in B^r$ corresponding to 
	a $G$-equivariant morphism $\vphi:E\to V_r$. Then $(b_1,\dots,b_r)$
	is a locally redundant generating tuple for $B$ if and only if 
	$\im(\vphi)\subseteq U_r-I_r$.
	
\smallskip

	Suppose for a moment that the claim is true.
	The action of 
	$G $ on $V_r$ restricts to $Z_r$ and to $p_i^{-1}(Z_{r-1})$ for every $i\in \{1,\dots,r\}$.
	Thus, by Lemma~\ref{LM:group-action-on-union}, there is a closed subscheme $W$ of $V_r$
	such that $W=Z_r\cup(\bigcap_{i=1}^r p_i^{-1}(Z_{r-1}))$ as sets  and the action of $G$
	restricts to $W$. Moreover, $Z_r$ is a closed subscheme 
	of $W$ and the complement $W-Z_r$ is precisely $I_r$. Since $I_r$
	is of finite type over $F$, we have $\dim I_r=\dim \overline{I_r}$ (the closure
	is in $V_r$), and so  
	\[
	\dim V_r-\dim W =\dim V_r- \max\{\dim \overline{I_r},\dim Z_r\}=\max\{ \dim  I_r ,\dim Z_r\}>d.
	\]
	Now,  
	\cite[Prop.~5.2]{First_2022_generators_of_alg_over_comm_ring} tells us that
	there is $\vphi\in \Hom_G(E,V_r)$ with $\im(\vphi)\subseteq V_r-W=U_r-I_r$.
	By our claim, this means that $B$ admits a locally redundant generating $r$-tuple,
	which is exactly what we want to show.
	
	{\it It remains to prove the claim.}
	Let $\vphi:E\to V_r$ and $b_1,\dots,b_r\in B$ be as in the claim.
	Observe first that the claim from the proof of 
	\cite[Prop.~4.3]{First_2022_generators_of_alg_over_comm_ring} 
	also tells us that $(b_1,\dots,\hat{b}_i,\dots,b_r)$
	generates $B$ if and only if the image of its corresponding $G$-equivariant
	morphism, which is $p_i\circ \vphi:E\to V_{r-1}$, lives in $U_{r-1}$,
	or equivalently, 
	$\im(\vphi)\subseteq V_r - p_i^{-1}(Z_{r-1})$ (here, $p_i$ is as in \eqref{EQ:Ir-dfn}).
	
	Suppose first that $\im(\vphi)\in U_r-I_r$. Then $\im(\vphi)\subseteq U_r$,
	which means that $(b_1,\dots,b_r)$ generates $B$.
	Let    $\frakm$ be a maximal ideal of $R$.
	By Proposition~\ref{PR:field-ext-gen}, 
	in order to show that  $(b_1+\frakm B,\dots,b_r+\frakm B)$ 
	generates $B(\frakm)$ redundantly
	over $R/\frakm$, it is enough to show that there is an $R/\frakm$-field
	$K$ such that  $B\otimes_R K$ is generated over $K$ by  a proper subtuple
	of $(b_1\otimes_R 1_K,\dots,b_r\otimes_R 1_K)$.
	We take $K$ to be an algebraic closure of $R/\frakm$.
	Let $E'$ be the pullback of $E\to\Spec R$ along $\Spec K\to \Spec R$.
	Then $E'$ is the $G$-torsor over $\Spec K$ corresponding to the $B\otimes_R K$, and the tuple
	$(b_1\otimes 1_K,\dots,b_r\otimes 1_K)\in (B\otimes_R K)^r$
	corresponds to the composition $E'\to E\xrightarrow{\vphi} V_r$, denoted
	$\vphi'$. As noted in the previous paragraph, if we could show
	that $\im(\vphi')\subseteq V_r-p_i^{-1}(Z_r)$ for some $i$,
	then it would follow that
	the image of $(b_1,\dots,\hat{b}_i,\dots,b_r)$ in $(B\otimes_R K)^r$,
	generates   $B\otimes_R K$,
	which is what we want.
	
	Since $K$ is algebraically closed, the closed points
	of $E'$ are its $K$-points (when viewed as a $K$-scheme).
	Let $x\in E'(K)$. 
	Then $\vphi'(x)\in \im(\vphi)\subseteq U_r\subseteq V_r-\bigcap_{i=1}^r p_i^{-1}(Z_{r-1})$.
	Thus, there is $i\in\{1,\dots,r\}$ such that $\vphi'(x)\notin p_i^{-1}(Z_{r-1})$.
	Since $\vphi'$ is $G$-equivariant, and since $G(K)$ acts transtively on
	$E'(K)$, it follows that $\vphi'(G(K))\subseteq V_r-p_i(Z_{r-1})$,
	and thus, $\vphi'(G)\subseteq V_r-p_i(Z_{r-1})$, as desired. 
	This completes the proof of
	the ``if'' part of the claim (and thus the proof of the theorem, which  only needs  this direction). 

	Suppose now that $(b_1,\dots,b_r)$ is a locally reudndant generating tuple for $B$.
	Then $\im(\vphi)\subseteq U_r$.
	Let $\frakm$ be a maximal ideal of $B$,
	let $b_i(\frakm)$ denote the image of $b_i$ in $B(\frakm)$
	and let $E_{\frakm}$ be the pullback of $E$
	along $\Spec R/\frakm\to \Spec R$.
	Then $(b_1(\frakm),\dots,b_r(\frakm))\in B(\frakm)^r$ corresponds to the
	composition $E_\frakm\to E\xrightarrow{\vphi} V_1$, denoted $\vphi_\frakm$.
	Since $(b_1,\dots,b_r)$ is locally redundant, there
	is  $i\in\{1,\dots,r\}$ such that
	$(b_1(\frakm),\dots,\hat{b}_i(\frakm),\dotsm,b_r(\frakm))$ generates $B(\frakm)$.
	By   our earlier observations, this is equivalent to $\im(\vphi_\frakm)\subseteq V_r-p_i(Z_{r-1})$.
	Since every closed point of $E$ is mapped to some closed point of $\Spec R$,
	it follows that $\vphi$ maps all the closed points of $E$
	into $U_r\cap (V_r-\bigcap_{i=1}^r p_i^{-1}(Z_{r-1}))=U_r-I_r$, so
	$\im(\vphi)\subseteq U_r-I_r$.
\end{proof}

\begin{remark}
	(i) It can happen that $\dim I^{(n)}_r> \dim Z^{(n)}_r$.
	For example, when $n>1$, we have that $I^{(n)}_{2}=U^{(n)}_2$ is open in $V_2^{(n)}$,
	so $\dim I^{(n)}_2=2n^2$, while $\dim Z^{(n)}_2=2n^2-(n-1)$.
	We suspect, however, that this is the only case where it happens.

	(ii) Theorem~\ref{TH:dim-Ir-conseq} remains true
	if, in the definition of $I_r^{(n)}$
	and $Z_r^{(n)}$, one replaces $\M$ with any finite-dimensional $F$-algebra
	$A$
	(or even a mutlialgebra in the sense of \cite{First_2022_generators_of_alg_over_comm_ring})
	and ``degree-$n$ Azumaya algebra over $R$'' with an ``$R$-form of $A$'', i.e.,
	an $R$-algebra $B$ for which there is a faithfully flat $R$-ring $S$
	such that $B\otimes_R S\cong A\otimes_F S$. The proof applies verbatim in this generality.
\end{remark}

\begin{cor}\label{CR:local-redundant-3}
	Let $R$ be a finitely generated ring over an infinite field $F$,
	let $d$ be the Krull dimension of $R$
	and let $B$ be an Azumaya algebra over $R$ of degree $3$.
	If $d\leq 7$,
	then $B$ 
	has a locally redundant generating $5$-tuple.
\end{cor}

\begin{proof}
	We have  $\dim I^{(3)}_5=19$ by Proposition~\ref{PR:dim-Ir}
	and $\dim Z^{(3)}_5=5\cdot 3^2-(5-1)(3-1)=37$ by
	\cite[Prop.~7.1(b)]{First_2022_generators_of_alg_over_comm_ring}.
	Thus,
	$d=7<45-37=\dim V_3^{(5)}-\max\{\dim I_5^{(3)},\dim Z_5^{(3)}\}$
	and the corollary follows from Theorem~\ref{TH:dim-Ir-conseq}.
\end{proof}

\begin{remark}
	Theorem~\ref{TH:dim-Ir-conseq} and Proposition~\ref{PR:dim-Ir}
	also imply that if $R$ has Krull dimension $\leq 1$
	and $B$ is an Azumaya algebra of degree $2$ over $R$,
	then $B$ admits locally redundant generating $3$-tuple. However, a stronger statement appears implicitly
	in \cite{First_2017_number_of_generators}. Indeed,  
	\cite[Thm.~1.2]{First_2017_number_of_generators} states
	that if $R$ is \emph{any} commutative noetherian ring
	of Krull dimension $d$, $B$ is \emph{any} $R$-algebra that is
	finitely generated
	as an $R$-module,
	and  $m\in\N$ is such that
	$B(\frakm) $ can be generated by $m$
	elements over $R/\frakm$ for every  $\frakm\in\Max R$, then $B$ admits a generating set $S$
	with $m+d$ elements. But the proof of this result 
	constructs a generating $(m+d)$-tuple $(b_1,\dots,b_{m+d})\in B^{m+d}$ with property
	that for every $\frakm\in\Max R$, just $m$ of the images
	$b_1+\frakm B,\dots,b_{m+d}+\frakm B\in B(\frakm)$  generate $B(\frakm)$
	over $R/\frakm$  (look at property (1) in the proof of {\it op.\ cit.}). 
	In particular, that generating tuple is locally redundant when $d>0$. When $B$ is Azumaya of degree
	$>2$ over $R$, we can take $m=2$ and get that $B$
	has a locally redundant generating $(d+2)$-tuple. 
	(This also implies Corollary~\ref{CR:local-redundant-3}
	under the stricter assumption $d\leq 3$, but without requiring that $R$ is finitely generated over an infinite field.)
\end{remark}

\bibliographystyle{plain}
\bibliography{MyBib_24_03}

\end{document}

%% file: amsart_header_18_05.tex
\usepackage{amsfonts}
\usepackage{amssymb}
\usepackage{amsmath}
\usepackage{hyperref}
\usepackage{mathrsfs}
\usepackage{centernot}
\usepackage{mathdots}
\usepackage{stmaryrd}
\usepackage[all]{xy}

\usepackage{mathtools}

\usepackage{color}

\newtheorem{thm}{Theorem}[section]
\newtheorem{lem}[thm]{Lemma}
\newtheorem{prp}[thm]{Proposition}
\newtheorem{cor}[thm]{Corollary}

\theoremstyle{definition}

\newtheorem{construction}[thm]{Construction}

\newtheorem{remark}[thm]{Remark}

\theoremstyle{plain}

\newcommand{\rem}[1]{}

\newcommand{\F}{\mathbb{F}}

\newcommand{\N}{\mathbb{N}}

\newcommand{\frakm}{{\mathfrak{m}}}

\newcommand{\bbA}{{\mathbb{A}}}

\newcommand{\bbP}{{\mathbb{P}}}

\newcommand{\vphi}{\varphi}

\newcommand{\idealof}{\unlhd} 

\newcommand{\suchthat}{\,:\,}
\newcommand{\where}{\,|\,}

\newcommand{\Trings}[1]{\left< #1 \right>}

\newcommand{\trings}[1]{\langle {#1} \rangle}

\newcommand{\floor}[1]{\lfloor {#1} \rfloor}

\DeclareMathOperator{\GL}{GL} %
\DeclareMathOperator{\Hom}{Hom} %
\DeclareMathOperator{\id}{id} %
\DeclareMathOperator{\im}{im} %

\DeclareMathOperator{\Max}{Max}
\DeclareMathOperator{\Mor}{Mor} %
\DeclareMathOperator{\Span}{span} %
\DeclareMathOperator{\Spec}{Spec} %
\DeclareMathOperator{\Stab}{Stab} %
\newcommand{\nGL}[2]{\mathrm{GL}_{#2}({#1})}
\newcommand{\nPGL}[2]{\mathrm{PGL}_{#2}({#1})}

\newcommand{\nMat}[2]{\mathrm{M}_{#2}(#1)}

\newcommand{\trans}{{\mathrm{t}}}

\newcommand{\units}[1]{{#1^\times}}

